\documentclass[12pt]{article}
\usepackage{amsmath, amsthm, amssymb}
\usepackage[top=1.25in, bottom=1.25in, left=1.0in, right=1.0in]{geometry}
\usepackage{hyperref}
\usepackage{color}
\usepackage{verbatim}
\usepackage{tikz,tkz-graph}
\usepackage{marginnote}
\usepackage[margin=1in]{caption}

\newcommand{\aside}[1]{\marginnote{\scriptsize{#1}}[0cm]}
\newcommand{\aaside}[2]{\marginnote{\scriptsize{#1}}[#2]}

\makeatletter
\newtheorem*{rep@theorem}{\rep@title}
\newcommand{\newreptheorem}[2]{
\newenvironment{rep#1}[1]{
 \def\rep@title{#2 \ref{##1}}
 \begin{rep@theorem}}
 {\end{rep@theorem}}}
\makeatother

\theoremstyle{plain}
\newtheorem{thm}{Theorem}
\newreptheorem{thm}{Theorem}
\newtheorem{prop}[thm]{Proposition}
\newreptheorem{prop}{Proposition}
\newtheorem{lem}[thm]{Lemma}
\newreptheorem{lem}{Lemma}

\newtheorem*{lemmaA}{Tashkinov's Lemma}
\newtheorem*{PEL}{Parallel Edge Lemma}
\newtheorem*{GS}{Goldberg--Seymour Conjecture}
\newtheorem*{main}{Theorem 20}
\newtheorem*{main2}{Theorem 11}
\newreptheorem{lemma}{Lemma}
\newtheorem{conj}[thm]{Conjecture}
\newreptheorem{conj}{Conjecture}
\newtheorem{cor}[thm]{Corollary}
\newreptheorem{cor}{Corollary}

\theoremstyle{definition}

\theoremstyle{remark}

\newcommand{\fancy}[1]{\mathcal{#1}}
\newcommand{\C}{\fancy{C}}

\newcommand{\W}{\fancy{W}}
\newcommand{\IN}{\mathbb{N}}

\newcommand{\T}{\fancy{T}}

\newcommand{\set}[1]{\left\{ #1 \right\}}

\newcommand{\setbs}[2]{\left\{ #1 : #2 \right\}}
\newcommand{\card}[1]{\left|#1\right|}
\newcommand{\size}[1]{\left\Vert#1\right\Vert}
\newcommand{\ceil}[1]{\left\lceil#1\right\rceil}
\newcommand{\floor}[1]{\left\lfloor#1\right\rfloor}

\newcommand{\irange}[1]{\left[#1\right]}

\newcommand{\parens}[1]{\left( #1 \right)}

\newcommand{\DefinedAs}{\mathrel{\mathop:}=}

\newcommand{\dclaw}[1]{d_{\text{claw}}\left( #1 \right)}

\newcommand{\vph}{\varphi}
\newcommand{\vphn}{\overline{\varphi}}

\newcommand{\claim}[2]{{\noindent\bf Claim #1.}~{\it #2}~~}
\newenvironment{claimproof}[1]{\par\noindent\underline{Proof:}\space#1}{\leavevmode\unskip\penalty9999
\hbox{}\nobreak\hfill\quad\hbox{$\qed$}}
\def\aftermath{\par\vspace{-\belowdisplayskip}\vspace{-\parskip}\vspace{-\baselineskip}}

\title{Short fans and the 5/6 bound for line graphs}
\author{Daniel W. Cranston\thanks{Department of Mathematics and Applied
Mathematics, Viriginia Commonwealth University, Richmond, VA;
\texttt{dcranston@vcu.edu}; 
The first author's research is partially supported by NSA Grant
H98230-15-1-0013.}
\and
Landon Rabern\thanks{
Franklin \& Marshall College, Lancaster, PA;
\texttt{landon.rabern@gmail.com}}
}

\begin{document}
\maketitle
\begin{abstract}
In 2011, the second author conjectured that every line graph $G$ satisfies
$\chi(G)\le \max\set{\omega(G),\frac{5\Delta(G)+8}{6}}$. This conjecture is best
possible, as shown by replacing each edge in a 5-cycle by $k$ parallel edges,
and taking the line graph. In this paper we prove the conjecture.
We also develop more general techniques and results that will likely be of
independent interest, due to their use in attacking the Goldberg--Seymour
conjecture.  
\end{abstract}

\section{Overview}

By \emph{graph} we mean multigraph without loops.  Our notation follows Diestel
\cite{diestel2010}\footnote{In particular, $\card{G}$ denotes $\card{V(G)}$ and
$\size{G}$ denotes $\card{E(G)}$.}.
In \cite{rabern2011strengthening}, the second author showed that $\chi(G)\le
\max\set{\omega(G), \frac{7\Delta(G)+10}{8}}$ for every line graph $G$.
In the same paper, he conjectured
that $\chi(G)\le \max\set{\omega(G),\frac{5\Delta(G)+8}{6}}$. This conjecture is best
possible, as shown by replacing each edge in a 5-cycle by $k$ parallel edges,
and taking the line graph. In this paper we prove the latter inequality.  Along
the way, we develop more general techniques and results that will likely be of
independent interest.  The main result of this paper is the following theorem.
\begin{main}[$\frac56$-Theorem]
If $Q$ is a line graph, then 
\[\chi(Q)\le \max\set{\omega(Q),\frac{5\Delta(Q)+8}{6}}.\]
\end{main}

For every graph $G$, we have $\chi'(G)\ge
\ceil{\frac{\size{G}}{\floor{\frac{\card{G}}{2}}}}$, since in any proper
edge-coloring each color class has size at most
$\floor{\frac{\card{G}}{2}}$.  Likewise, the same bound holds for any subgraph $H$.  
Thus $\chi'(G)\ge \max_{H\subseteq G}\ceil{\frac{\size{H}}{\floor{\frac{\card{H}}{2}}}}$
(where the max is over all subgraphs $H$ with at least two vertices).
For convenience, we let $\W(G)\DefinedAs\max_{H\subseteq
G}\ceil{\frac{\size{H}}{\floor{\frac{\card{H}}{2}}}}$.
Goldberg~\cite{goldberg1973,GoldbergJGT} and
Seymour~\cite{seymour1979a,seymour1979b} each conjectured that this lower bound
holds with equality, whenever $\chi'(G)>\Delta(G)+1$.
\begin{GS}
Every graph $G$ satisfies
\[
\chi'(G)\le\max\{\W(G), \Delta(G)+1\}.
\]
\end{GS}
The Goldberg--Seymour Conjecture is the major open problem in the area of
edge-coloring multigraphs.
Most of our work goes toward proving the following
intermediate result, in Section~\ref{sec:thin}.  This theorem is a weakened
version of both the Goldberg--Seymour Conjecture and our main result, the
$\frac56$-Theorem.
\begin{main2}[Weak $\frac56$-Theorem]
If $Q$ the line graph of a graph $G$, then 
\[\chi(Q)\le \max\set{\W(G),\Delta(G)+1,\frac{5\Delta(Q)+8}{6}}.\]
\end{main2}
Finally, in Section~\ref{sec:final} we show that the Weak $\frac56$-Theorem
does indeed imply the $\frac56$-Theorem.
To conclude, in Section~\ref{sec:strong-reed} we prove strengthenings of Reed's
Conjecture for line graphs that follow from the general lemmas we prove earlier
in the paper.

\section{Tashkinov Trees}
A graph $G$ is \emph{elementary}\aside{elementary graph} if $\chi'(G)=\W(G)$. 
Let $[k]$ denote $\{1,\ldots,k\}$.
For a path or cycle $Q$, let \emph{$\ell(Q)$} denote the length of $Q$.
A graph $G$ is \emph{critical}\aaside{$\ell(Q)$ critical}{.1in} if $\chi'(G-e)
< \chi'(G)$ for all $e \in E(G)$. 
For a graph $G$ and a partial $k$-edge-coloring $\varphi$, for each vertex $v\in
V(G)$, let $\varphi(v)$ denote the set of colors used in $\varphi$ on edges
incident to $v$.  Let $\vphn(v)=[k]\setminus\varphi(v)$.  A color $c$ is
\emph{seen}\aside{seen} by a vertex $v$ if $c\in \varphi(v)$ and $c$ is
\emph{missing}\aside{missing} by $v$
if $c\in\vphn(v)$.
Given a partial $k$-edge-coloring $\varphi$, a set $W\subseteq V(G)$ is
\emph{elementary}\aside{elementary set} with respect to $\varphi$ (henceforth,
\emph{w.r.t.~$\varphi$}) if each color in $[k]$ is
missing at no more than one vertex of $W$.  More formally, $\vphn(u)\cap
\vphn(v)=\emptyset$ for all distinct $u,v\in W$.
A \emph{defective color}\aaside{defective color}{-.1in} for a set $X\subseteq V(G)$
(w.r.t.~$\varphi$) is a color
used on more than one edge from $X$ to $V(G) \setminus X$.  
A set $X$ is \emph{strongly closed}\aside{strongly closed} 
w.r.t.~$\varphi$ if $X$ has no defective color.
Elementary and strongly closed sets are of particular interest because of the
following theorem, proved implicitly by Andersen~\cite{andersen1977edge} and
Goldberg~\cite{GoldbergJGT}; see also~\cite[Theorem 1.4]{SSTF}.
%

\begin{thm}
\label{elementary}
Let $G$ be a graph with $\chi'(G)=k+1$ for some integer $k\ge \Delta(G)$.  If
$G$ is critical, then $G$ is elementary if and only if there exists $uv\in E(G)$,
a $k$-edge-coloring $\vph$ of $G-uv$, and a set $X$ with $u,v\in X$ such
that $X$ is both elementary and strongly closed w.r.t.~$\varphi$.
\end{thm}

A \emph{Tashkinov tree}\aside{Tashkinov tree} w.r.t.~$\varphi$ is a sequence
$v_0, e_1, v_1,
e_2,\ldots, v_{t-1},e_t,v_t$ such that all $v_i$ are distinct, $e_i=v_jv_i$ and
$\vph(e_i)\in \vphn(v_\ell)$ for some $j$ and $\ell$ with $0\le j< i$ and $0\le
\ell < i$.  
A \emph{Vizing fan}\aside{Vizing fan} (or simply \emph{fan}) is a Tashkinov tree that induces a
star.  Tashkinov trees are of interest because of the following lemma. 

\begin{lemmaA}
Let $G$ be a graph with $\chi'(G)=k+1$, for some integer $k\ge \Delta(G)+1$ and
choose $e\in E(G)$ such that $\chi'(G-e)<\chi'(G)$.  Let $\varphi$ be a
$k$-edge-coloring of $G-e$.  If $T$ is a Tashkinov tree w.r.t.~$\varphi$ and
$e$, then $V(T)$ is elementary w.r.t.~$\varphi$.
\end{lemmaA}

In view of Theorem~1 and Tashkinov's Lemma, to prove that a graph $G$ is elementary,
it suffices to find an edge $e$, a $k$-edge-coloring $\vph$ of $G-e$, and a
Tashkinov tree $T$ containing $e$ such that $V(T)$ is strongly closed.
This motivates our next two lemmas.  But first, we need a few more definitions.

Let $t(G)$\aside{$t(G)$} be the maximum number of vertices in a Tashkinov tree
over all $e \in E(G)$
and all $k$-edge-colorings $\vph$ of $G - e$.  Let $\T(G)$\aside{$\T(G)$} be
the set of all triples $(T,e,\vph)$ such that $e \in E(G)$, $\vph$ is a
$k$-edge-coloring of $G-e$, and $T$ is a Tashkinov tree with respect to $e$ and
$\vph$ with $\card{T} = t(G)$.  Notice that, by definition, we have $\T(G) \ne \emptyset$.
For a $k$-edge-coloring $\vph$ of $G-e$, a maximal Tashkinov tree
starting with $e$ may not be unique.  However, if $T_1$ and $T_2$ are both such
trees, then it is easy to show that $V(T_1)\subseteq V(T_2)$; by symmetry, also
$V(T_2)\subseteq V(T_1)$, so $V(T_1)=V(T_2)$.
Let $G$ be a critical graph with $\chi'(G) = k+1$ for some integer $k \ge \Delta(G) + 1$. 
Let $\varphi$ be a $k$-edge-coloring of $G - e_0$ for some $e_0 \in E(G)$.  
For $v \in V(G)$ and colors $\alpha, \beta$, let $P_v(\alpha,
\beta)$\aside{$P_v(\alpha,\beta)$} be the
maximal connected subgraph of $G$ that contains $v$ and is induced by edges with color
$\alpha$ or $\beta$.  So $P_v(\alpha, \beta)$ is a path or a cycle.
For a $k$-edge-coloring $\vph$ of $G-v_0v_1$, we often let
$P=P_{v_1}(\alpha,\beta)$ for some $\alpha\in\vphn(v_0)$ and
$\beta\in\vphn(v_1)$.  
Clearly $P$ must end at $v_0$ (or we can swap colors $\alpha$ and $\beta$ on
$P$ and color $v_0v_1$ with $\alpha$), so let $v_1,\ldots,v_r,v_0$ denote the
vertices of $P$ in order. 
To \emph{rotate the $\alpha,\beta$ coloring on $P\cup\{v_0v_1\}$ by
one}\aside{rotate coloring}, we
uncolor $v_1v_2$ and use its color on $v_0v_1$.  To \emph{rotate the
$\alpha,\beta$ coloring on $P\cup\{v_0v_1\}$ by $j$}, we rotate the
$\alpha,\beta$ coloring by one $j$ times in succession.
(When we do not specify $j$, we allow $j$ to take any value from $1$ to
$r$.)  


\begin{lem}
\label{FreeColorsLemma}
Let $G$ be a non-elementary critical graph with $\chi'(G) = k+1$ for some integer
$k \ge \Delta(G) + 1$.  For every $v_0v_1 \in E(G)$, $k$-edge-coloring $\vph$
of $G-v_0v_1$, $\alpha \in \vphn(v_0)$, and $\beta \in
\vphn(v_1)$, we have $\card{P_{v_1}(\alpha, \beta)} < t(G)$.
\end{lem}
\begin{proof}
Suppose the lemma is false and choose $v_0v_1 \in E(G)$, a $k$-edge-coloring
$\vph$ of $G-v_0v_1$, $\alpha \in \vphn(v_0)$, and $\beta \in \vphn(v_1)$, such
that $\card{P_{v_1}(\alpha, \beta)} \ge t(G)$.  Let $P = P_{v_1}(\alpha, \beta)$; see
Figure~\ref{fig:FreeColors}.  
Let $(T, v_0v_1, \vph)$ be a Tashkinov tree that
begins with edges $v_0v_1, v_1v_2, \ldots, v_{r-1}v_r$.  Now $V(T)=V(P)$ since
$t(G) \ge \card{T} \ge \card{P} \ge t(G)$.
By hypothesis $G$ is non-elementary, so Theorem~\ref{elementary} implies that
$V(T)$ is not strongly closed; thus, $T$ has a defective color $\delta$ with
respect to $\vph$.  Choose $\tau\in \vphn(v_2)$. Let $Q = P_{v_2}(\tau, \delta)$.
Since $T$ is maximal, $\delta$ is not missing at any vertex of $T$, and
since $V(T)$ is elementary, $\tau$ is not missing at any vertex of $T$ other 
than $v_2$.  As a result, $Q$ ends outside $V(T)$.  Now $Q$ could leave
$V(T)$ and re-enter it repeatedly, but $Q$ ends outside $V(T)$, so there is a
last vertex $w \in V(Q) \cap V(T)$; say $Q$ ends at $z \in V(G)\setminus V(T)$.
 Let $\pi \notin \{\alpha, \beta\}$ be a color missing at $w$.  
Since $\tau\in\vphn(v_2)$ and $\pi\in\vphn(w)$ and $\card{T} = t(G)$, no edge
colored $\tau$ or $\pi$ leaves $V(T)$.  So we can swap $\tau$ and $\pi$ on
every edge in $G - V(T)$ without changing the fact that $T$ is a Tashkinov tree
with $\card{T} = t(G)$.  After swapping $\tau$ and $\pi$, we swap $\delta$ and $\pi$
on the subpath of $Q$ from $w$ to $z$. Since $\pi$ is missing at $w$, the
$\delta-\pi$ path starting at $z$ must end at $w$.  Now $\delta$ is missing at
$w$, but $\delta$ was defective in $\vph$, so some other edge $e$ colored
$\delta$ still leaves $V(T)$. Adding $e$ gets a larger Tashkinov tree, which is
a contradiction.
\end{proof}

\begin{figure}
\centering
\begin{tikzpicture}[scale = 12, line width=.06em]
\tikzstyle{VertexStyle} = []
\tikzstyle{EdgeStyle} = []
\tikzstyle{edgeLeft} = [style={bend left=10}]
\tikzstyle{edgeRight} = [style={bend right=10}]
\tikzstyle{labeledStyle}=[shape = circle, minimum size = 6pt, inner sep = 1.2pt]
\tikzstyle{labeledStyle2}=[shape = circle, minimum size = 6pt, inner sep =
1.2pt, draw]
\tikzstyle{unlabeledStyle}=[shape = circle, minimum size = 6pt, inner sep = 1.2pt, draw, fill]
\draw (0,.5) circle (.20cm);
\Vertex[style = labeledStyle2, x = -0.035, y = 0.333, L = \small {$\alpha$}]{v0}
\Vertex[style = labeledStyle2, x = 0.035, y = 0.338, L = \small {$\beta$}]{v1}
\Vertex[style = unlabeledStyle, x = -.035, y = 0.303, L = \tiny {}]{v2}
\Vertex[style = labeledStyle, x = -0.035, y = 0.275, L = \small {$v_0$}]{v3}
\Vertex[style = unlabeledStyle, x = 0.035, y = 0.303, L = \tiny {}]{v4}
\Vertex[style = labeledStyle, x = 0.035, y = 0.275, L = \small {$v_1$}]{v5}
\Vertex[style = unlabeledStyle, x = 0.0995, y = 0.325, L = \small {}]{v6}
\Vertex[style = labeledStyle, x = 0.10000, y = 0.295, L = \small {$v_2$}]{v7}
\Vertex[style = labeledStyle2, x = 0.10000, y = 0.355, L = \small {$\tau$}]{v8}
\draw[line width=.06em] plot [smooth] coordinates {(.1925,.45) (.33,.42)
(.36,.46) (.36,.54) (.33,.58) (.1925,.55)}; 
\Vertex[style = labeledStyle, x = 0.335, y = 0.5, L = \small {$\tau,\delta$}]{v9}
\Vertex[style = labeledStyle, x = 0.22, y = 0.425, L = \small {$\delta$}]{v10}
\Vertex[style = labeledStyle, x = 0.22, y = 0.575, L = \small {$\delta$}]{v11}
\Vertex[style = unlabeledStyle, x = 0.1925, y = 0.45, L = \small {}]{v12}
\Vertex[style = unlabeledStyle, x = 0.1925, y = 0.55, L = \small {}]{v13}
\draw[line width=.06em] plot [smooth] coordinates {
(-.1925,.55) (-.33,.58) (-.36,.54) (-.36,.46) (-.33,.42) (-.26,.4325)}; 
\Vertex[style = labeledStyle, x = -0.335, y = 0.5, L = \small
{$\tau,\delta$}]{v14}
\Vertex[style = unlabeledStyle, x = -.26, y = 0.4325, L = \small {}]{v15}
\Vertex[style = labeledStyle, x = -.26, y = 0.405, L = \small {$z$}]{v15b}
\Vertex[style = unlabeledStyle, x = -0.1925, y = 0.55, L = \small {}]{v16}
\Vertex[style = labeledStyle, x = -0.22, y = 0.575, L = \small {$\delta$}]{v17}
\Vertex[style = labeledStyle, x = -0.2125, y = 0.535, L = \small {$w$}]{v18}
\Vertex[style = labeledStyle2, x = -0.163, y = 0.55, L = \small {$\pi$}]{v19}
\end{tikzpicture}
\caption{$Q\cup(P_{v_1}(\alpha,\beta)+v_0v_1)$ in the proof of
Lemma~\ref{FreeColorsLemma}. 
By recoloring some of $wQz$, we form a larger Tashkinov tree.
Colors that are circled are missing at the indicated vertex.
\label{fig:FreeColors}}
\end{figure}
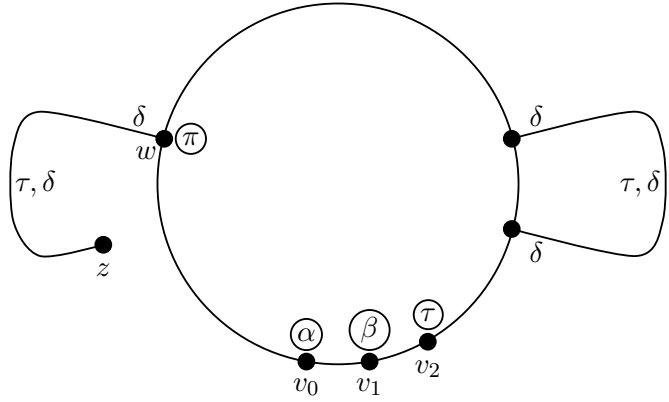

\section{Short vertices and long vertices}
\label{sec:short}
A vertex $v \in V(G)$ is \emph{short}\aaside{short vertex}{-.0in} if every
Vizing fan rooted at $v$ (taken over all $k$-colorings of $G-e$, over all edges
$e$ incident to $v$) has at most 3 vertices, including $v$.  Otherwise, $v$ is
\emph{long}\aaside{long vertex}{-.1in}.  Let $\nu(T)$\aaside{$\nu(T)$}{.1in} be the
number of long vertices in a Tashkinov tree $T$.

Now we can outline our proof of the $\frac56$-Conjecture.  We will show in
Section~\ref{sec:final} (and at the end of Section~\ref{sec:thin})
that the $\frac56$-Conjecture is implied
by the Goldberg--Seymour Conjecture.  More precisely, if $Q$ is the line graph
of graph $G$ and
$\chi(Q)=\chi'(G)\le\max\set{\W(G),\Delta(G)+1}$, then also $\chi(Q)\le
\max\{\omega(Q),\frac{5\Delta(Q)+8}6\}$.  So here it suffices to prove the bound
$\chi'(G)\le\max\set{\W(G), \Delta(G)+1, \frac{5\Delta(Q)+8}6}$.  We consider
cases based on $\nu(T)$, for some Tashkinov tree $T\in \T(G)$.

In the present section, we show that if $G$ has a maximum Tashkinov tree $T$ that contains
no long vertices, i.e., $\nu(T)=0$, then $G$ is elementary.  In fact, Lemma~7
implies that the same is true when $\nu(T)=1$.  In the proof of
Theorem~\ref{mainhelper}, we show that if $G$ is a minimal counterexample to
the $\frac56$-Conjecture, then every long vertex $v$ has
$d(v)<\frac34\Delta(G)$.  This implies that $\nu(T)< 4$, since otherwise
the number of colors missing at vertices of $T$ is more than
$4(k-\frac34\Delta(G))>k$, which contradicts that $V(T)$ is elementary.
So it remains to consider the case $\nu(T)\in\{2,3\}$.

In Section~\ref{sec:thin}, we introduce \emph{$k$-thin graphs};
these are essentially graphs for which $\mu(G)$ is not too large.  Using a lemma
from~\cite{rabern2011strengthening}, we show that every minimal counterexample
to the $\frac56$-Conjecture must be $k$-thin.  We then extend the ideas of the
present section to handle the case when $\nu(T)\in\{2,3\}$.  Much like when
$\nu(T)\ge 4$, we show that $T$ has too many colors missing at its vertices to
be elementary. 

Short vertices were introduced in~\cite{CKPS}, where they were motivated by a
version of the following lemma in the context of proving a strengthening of
Reed's Conjecture for line graphs.

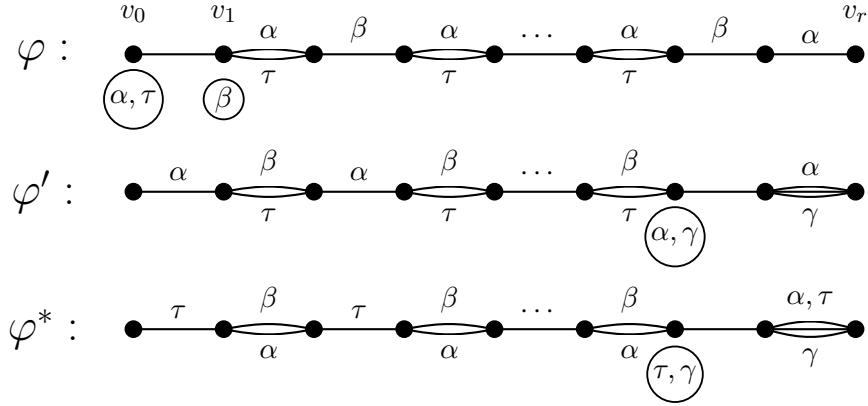
\begin{figure}
\centering
\begin{tikzpicture}[scale = 12, line width=.06em]
\tikzstyle{VertexStyle} = []
\tikzstyle{EdgeStyle} = []
\tikzstyle{edgeLeft} = [style={bend left=10}]
\tikzstyle{edgeRight} = [style={bend right=10}]
\tikzstyle{labeledStyle}=[shape = circle, minimum size = 6pt, inner sep = 1.2pt]
\tikzstyle{labeledStyle2}=[shape = circle, minimum size = 6pt, inner sep =
1.2pt, draw]
\tikzstyle{unlabeledStyle}=[shape = circle, minimum size = 6pt, inner sep = 1.2pt, draw, fill]
\Vertex[style = labeledStyle, x = 0.150, y = 0.850, L=\Large{$\vph:$}]{vphi}
\Vertex[style = unlabeledStyle, x = 0.250, y = 0.850, L = \tiny {}]{v0}
\Vertex[style = unlabeledStyle, x = 0.350, y = 0.850, L = \tiny {}]{v1}
\Vertex[style = unlabeledStyle, x = 0.450, y = 0.850, L = \tiny {}]{v2}
\Vertex[style = unlabeledStyle, x = 0.550, y = 0.850, L = \tiny {}]{v3}
\Vertex[style = unlabeledStyle, x = 0.650, y = 0.850, L = \tiny {}]{v4}
\Vertex[style = unlabeledStyle, x = 0.750, y = 0.850, L = \tiny {}]{v5}
\Vertex[style = unlabeledStyle, x = 0.850, y = 0.850, L = \tiny {}]{v6}
\Vertex[style = unlabeledStyle, x = 0.950, y = 0.850, L = \tiny {}]{v7}
\Vertex[style = unlabeledStyle, x = 1.050, y = 0.850, L = \tiny {}]{v8}
\Vertex[style = labeledStyle, x = 0.250, y = 0.895, L = \small {$v_0$}]{v9}
\Vertex[style = labeledStyle, x = 0.350, y = 0.895, L = \small {$v_1$}]{v10}
\Vertex[style = labeledStyle, x = 1.050, y = 0.895, L = \small {$v_r$}]{v16}
\Vertex[style = labeledStyle2, x = 0.250, y = 0.800, L = \small {$\alpha, \tau$}]{v17}
\Vertex[style = labeledStyle2, x = 0.350, y = 0.800, L = \small {$\beta$}]{v18}
\Edge[label = \small {}, labelstyle={auto=right, fill=none}](v0)(v1)
\Edge[style = edgeLeft, label = \small {$\tau$}, labelstyle={auto=left}](v2)(v1)
\Edge[style = edgeRight, label = \small {$\alpha$}, labelstyle={auto=right}](v2)(v1)
\Edge[label = \small {$\beta$}, labelstyle={auto=left, fill=none}](v2)(v3)
\Edge[style = edgeLeft, label = \small {$\tau$}, labelstyle={auto=left}](v4)(v3)
\Edge[style = edgeRight, label = \small {$\alpha$}, labelstyle={auto=right}](v4)(v3)
\Edge[label = \small {$\cdots$}, labelstyle={auto=left, fill=none}](v4)(v5)
\Edge[style = edgeLeft, label = \small {$\tau$}, labelstyle={auto=left}](v6)(v5)
\Edge[style = edgeRight, label = \small {$\alpha$}, labelstyle={auto=right}](v6)(v5)
\Edge[label = \small {$\beta$}, labelstyle={auto=left, fill=none}](v6)(v7)
\Edge[label = \small {$\alpha$}, labelstyle={auto=right, fill=none}](v8)(v7)

\begin{scope}[yshift=-.06in]
\Vertex[style = labeledStyle, x = 0.150, y = 0.850, L=\Large{$\vph':$}]{vphi}
\Vertex[style = unlabeledStyle, x = 0.250, y = 0.850, L = \tiny {}]{v0}
\Vertex[style = unlabeledStyle, x = 0.350, y = 0.850, L = \tiny {}]{v1}
\Vertex[style = unlabeledStyle, x = 0.450, y = 0.850, L = \tiny {}]{v2}
\Vertex[style = unlabeledStyle, x = 0.550, y = 0.850, L = \tiny {}]{v3}
\Vertex[style = unlabeledStyle, x = 0.650, y = 0.850, L = \tiny {}]{v4}
\Vertex[style = unlabeledStyle, x = 0.750, y = 0.850, L = \tiny {}]{v5}
\Vertex[style = unlabeledStyle, x = 0.850, y = 0.850, L = \tiny {}]{v6}
\Vertex[style = unlabeledStyle, x = 0.950, y = 0.850, L = \tiny {}]{v7}
\Vertex[style = unlabeledStyle, x = 1.050, y = 0.850, L = \tiny {}]{v8}
\Edge[label = \small {$\alpha$}, labelstyle={auto=left, fill=none}](v0)(v1)
\Edge[style = edgeLeft, label = \small {$\tau$}, labelstyle={auto=left}](v2)(v1)
\Edge[style = edgeRight, label = \small {$\beta$}, labelstyle={auto=right}](v2)(v1)
\Edge[label = \small {$\alpha$}, labelstyle={auto=left, fill=none}](v2)(v3)
\Edge[style = edgeLeft, label = \small {$\tau$}, labelstyle={auto=left}](v4)(v3)
\Edge[style = edgeRight, label = \small {$\beta$}, labelstyle={auto=right}](v4)(v3)
\Edge[label = \small {$\cdots$}, labelstyle={auto=left, fill=none}](v4)(v5)
\Edge[style = edgeLeft, label = \small {$\tau$}, labelstyle={auto=left}](v6)(v5)
\Edge[style = edgeRight, label = \small {$\beta$}, labelstyle={auto=right}](v6)(v5)
\Edge[label = \small {}, labelstyle={auto=left, fill=none}](v6)(v7)
\Edge[label = \small {$\alpha$}, labelstyle={auto=right, fill=none}](v8)(v7)
\Vertex[style = labeledStyle2, x = 0.850, y = 0.800, L = \small {$\alpha, \gamma$}]{v17}
\Edge[style = edgeLeft, label = \small {$\gamma$}, labelstyle={auto=left}](v8)(v7)
\Edge[style = edgeRight, label = \small {$\alpha$}, labelstyle={auto=right}](v8)(v7)
\end{scope}

\tikzstyle{edgeLefter} = [style={bend left=15}]
\tikzstyle{edgeRighter} = [style={bend right=15}]
\begin{scope}[yshift=-.12in]
\Vertex[style = labeledStyle, x = 0.150, y = 0.850, L=\Large{$\vph^*:$}]{vphi}
\Vertex[style = unlabeledStyle, x = 0.250, y = 0.850, L = \tiny {}]{v0}
\Vertex[style = unlabeledStyle, x = 0.350, y = 0.850, L = \tiny {}]{v1}
\Vertex[style = unlabeledStyle, x = 0.450, y = 0.850, L = \tiny {}]{v2}
\Vertex[style = unlabeledStyle, x = 0.550, y = 0.850, L = \tiny {}]{v3}
\Vertex[style = unlabeledStyle, x = 0.650, y = 0.850, L = \tiny {}]{v4}
\Vertex[style = unlabeledStyle, x = 0.750, y = 0.850, L = \tiny {}]{v5}
\Vertex[style = unlabeledStyle, x = 0.850, y = 0.850, L = \tiny {}]{v6}
\Vertex[style = unlabeledStyle, x = 0.950, y = 0.850, L = \tiny {}]{v7}
\Vertex[style = unlabeledStyle, x = 1.050, y = 0.850, L = \tiny {}]{v8}
\Edge[label = \small {$\tau$}, labelstyle={auto=left, fill=none}](v0)(v1)
\Edge[style = edgeLeft, label = \small {$\alpha$}, labelstyle={auto=left}](v2)(v1)
\Edge[style = edgeRight, label = \small {$\beta$}, labelstyle={auto=right}](v2)(v1)
\Edge[label = \small {$\tau$}, labelstyle={auto=left, fill=none}](v2)(v3)
\Edge[style = edgeLeft, label = \small {$\alpha$}, labelstyle={auto=left}](v4)(v3)
\Edge[style = edgeRight, label = \small {$\beta$}, labelstyle={auto=right}](v4)(v3)
\Edge[label = \small {$\cdots$}, labelstyle={auto=left, fill=none}](v4)(v5)
\Edge[style = edgeLeft, label = \small {$\alpha$}, labelstyle={auto=left}](v6)(v5)
\Edge[style = edgeRight, label = \small {$\beta$}, labelstyle={auto=right}](v6)(v5)
\Edge[label = \small {}, labelstyle={auto=left, fill=none}](v6)(v7)
\Vertex[style = labeledStyle2, x = 0.850, y = 0.800, L = \small {$\tau, \gamma$}]{v17}
\Edge[style = edgeLefter, label = \small {$\gamma$}, labelstyle={auto=left}](v8)(v7)
\Edge[](v8)(v7)
\Edge[style = edgeRighter, label = \small {$\alpha, \tau$}, labelstyle={auto=right}](v8)(v7)
\end{scope}

\end{tikzpicture}
\caption{Edge-colorings $\vph$, $\vph'$, and $\vph^*$ in the proof of
the \hyperref[SpecialPath]{Parallel Edge Lemma}.\label{fig:PEL}}
\end{figure}

\begin{PEL}\label{SpecialPath}
Let $G$ be a critical graph with $\chi'(G) = k+1$ for some integer $k \ge
\Delta(G) + 1$.  Let $\vph$ be a $k$-edge-coloring of $G-v_0v_1$.  Choose
$\alpha \in \vphn(v_0)$ and $\beta \in \vphn(v_1)$.  Let $P = v_1v_2\cdots v_r$
be an $\alpha,\beta$ path with edges $e_i = v_iv_{i+1}$ for all $i\in[r-1]$. 
If $v_i$ is short for all odd $i$, then for each $\tau \in \vphn(v_0)$ 
and for all odd $i\in[r-1]$ there are edges $f_i = v_iv_{i+1}$ such that 
$\vph(f_i) = \tau$. 
\end{PEL}
\begin{proof}
Suppose not and choose a counterexample minimizing $r$.  By minimality of
$r$, we have $\vph(v_{r-1}v_r) = \alpha$ and we have $f_i = v_iv_{i+1}$ for
all odd $i\in[r-2]$ such that 
$\vph(f_i) = \tau$; 
see Figure~\ref{fig:PEL}.  Swap $\alpha$ and $\beta$ on $e_i$ for all
$i\in [r-3]$ and then
color $v_0v_1$ (call this edge $e_0$) with $\alpha$ and uncolor $e_{r-2}$.  Let
$\vph'$ be the resulting coloring.  Since $k \ge \Delta(G) + 1$, some color
other than $\alpha$ is missing at $v_{r-2}$; let $\gamma$ be such a color.  Now 
$v_{r-1}$ is short since $r-1$ is odd (since $P$ starts and ends with
$\alpha$), so there is an edge $e = v_{r-1}v_r$ with $\vph'(e) = \gamma$.  
Swap $\tau$ and $\alpha$ on $e_i$ for all $i$ with $0 \le i \le r-3$ to get a
new coloring $\vph^*$.  Now $\gamma$ and $\tau$ are both missing at $v_{r-2}$
in $\vph^*$.  Since $v_{r-1}$ is short, the fan with $v_{r-2}, v_{r-1}, v_r$
and $e$ implies that there is an edge $f_{r-1} = v_{r-1}v_r$ with
$\vph^*(f_{r-1}) = \tau$.  But we have never recolored $f_{r-1}$, so
$\vph(f_{r-1})=\tau$, which is a contradiction.
\end{proof}

\begin{lem}\label{ZeroNonSpecial}
Let $G$ be a non-elementary critical graph with $\chi'(G) = k+1$ for some
integer $k \ge \Delta(G) + 1$.  Choose $(T, v_0v_1, \vph) \in \T(G)$ for some
$v_0v_1 \in E(G)$.  Choose $\alpha \in \vphn(v_0)$ and $\beta \in \vphn(v_1)$ and
let $P = P_{v_1}(\alpha, \beta)$.  Now $P$ contains a long vertex. 
In particular, $\nu(T) \ge 1$.
\label{lem2}
\end{lem}
\begin{proof}
Suppose every vertex of $P$ is short.  Applying the
\hyperref[SpecialPath]{Parallel Edge Lemma} to $P$
shows that for every $\tau \in \vphn(v_0)$, there is an edge in $T$ colored
$\tau$ incident to every $v \in V(P - v_0)$.  The same is also true of every 
other color missing at some vertex of $P$; to see this, we rotate the
$\alpha,\beta$ coloring of $P\cup\{v_0v_1\}$ and repeat the same argument. 
Hence $V(P) = V(T)$, which contradicts Lemma \ref{FreeColorsLemma}.
\end{proof}

\begin{thm}
\label{AllSpecialImpliesElementary}
If $G$ is a critical graph in which every vertex is short, then
\[\chi'(G) \le \max \set{\W(G), \Delta(G) + 1}.\]
\end{thm}
\begin{proof}
Suppose not and let $G$ be a counterexample. 
Let $k = \chi'(G) - 1$, and note that $k \ge \Delta(G) + 1$.  
Since $\T(G) \ne \emptyset$, by applying Lemma \ref{ZeroNonSpecial} we conclude
that $G$ is elementary.  Hence $\chi'(G) = \W(G)$, which is a
contradiction.
\end{proof}

For a path $Q$, recall that $\ell(Q)$ denotes the length of $Q$.
For $x,y \in V(Q)$, let $xQy$ denote the subpath of $Q$ with
end vertices $x$ and $y$, and let $d_Q(x,y) = \ell(xQy)$, i.e., the distance
from $x$ to $y$ along $Q$.

\begin{lem}\label{TauEscape}
Let $G$ be a critical graph with $\chi'(G) = k+1$ for some integer $k \ge \Delta(G) + 1$.
Let $\vph$ be a $k$-edge-coloring of $G-v_0v_1$. Choose $\alpha \in \vphn(v_0)$
and $\beta \in \vphn(v_1)$ and let $C = P_{v_1}(\alpha, \beta) + v_0v_1$.  If
$\tau \in \vphn(x)$ for some $x \in V(C)$ and there is a $\tau$-colored edge
from $y \in V(C)$ to $w \in V(G) \setminus V(C)$, then $C$ has a subpath $Q$
with long endpoints $z_1,z_2$ such that $x \in V(Q)$, $y \not \in V(Q-z_1-z_2)$
and the distance from $x$ to $z_i$ along $Q$ is odd for each $i \in \irange{2}$. 
Moreover, for each $i \in \irange{2}$, there are no $\tau$-colored edges
between $z_i$ and its neighbors along $C$.
\end{lem}
\begin{proof}
Let $G$, $\alpha$, $\beta$, $\tau$, $x$, and $y$ be as in the statement of the
lemma.  Choose $z_1$ (resp. $z_2$) to be the first vertex at an odd distance
from $x$ along $C$ in the clockwise (resp. counterclockwise) direction with no
incident $\tau$-colored edge parallel to some edge of $C$.  
Let $Q$ be the subpath of $C$ with endpoints $z_1$ and $z_2$ that contains $x$.
By the choice of $z_1$ each vertex $w$ between $x$ and $z_1$ with $d_Q(x,w)$ odd
has a $\tau$-colored edge parallel to some edge of $C$.  The presence of these
edges implies the same for each $w$ for which $d_Q(x,w)$ is even.  By the proof
of 
the \hyperref[SpecialPath]{Parallel Edge Lemma}, 
$z_1$ must be long, since otherwise it would have an
incident $\tau$-colored edge parallel to some edge of $C$.  The same argument
applies to $z_2$.
%
\end{proof}

\section{Thin graphs}
\label{sec:thin}
Let $G$ be a critical graph with $\chi'(G) = k+1$ for some integer $k \ge \Delta(G) + 1$.
For vertices $x \in V(G)$ and $S \subseteq V(G) \setminus \set{x}$, we say that
$x$ is \emph{$S$-short}\aside{$S$-short} if every Vizing fan $F$ rooted at $x$
with $(\{x\}\cup S) \subseteq V(F)$, has $\card{F} \le 3$ (with respect to any
$k$-edge-coloring of $G-xy$).  Otherwise, $x$ is
\emph{$S$-long}\aside{$S$-long}.  For brevity,
when $S = \set{y}$, we may write $y$-short instead of $\set{y}$-short.  It is
worth noting that in the \hyperref[SpecialPath]{Parallel Edge Lemma} we can
weaken the hypothesis that $v_i$ is short for all odd $i$ to require only that
$v_i$ is $v_{i-1}$-short for all odd $i$, since this is what we use in the proof.  

A graph $G$ is \emph{$k$-thin}\aside{$k$-thin} if $\mu(G) < 2k - d(x) - d(y)$
for all distinct long $x,y \in V(G)$.  In the proof of
Theorem~\ref{mainhelper}, we will show that every minimum counterexample to
the $\frac56$-Conjecture must be $k$-thin.

\begin{lem}
\label{NonSpecialsInThinAreAtEvenDistance}
Let $G$ be a $k$-thin, critical graph with $\chi'(G) = k+1$ for some integer $k \ge \Delta(G) + 1$.
Let $\vph$ be a $k$-edge-coloring of $G-v_0v_1$. Choose $\alpha \in \vphn(v_0)$
and $\beta \in \vphn(v_1)$ and let $C = P_{v_1}(\alpha, \beta) + v_0v_1$.
Let $Q$ be a subpath of $C$ with long end vertices.  If all internal vertices
of $Q$ are short and $2 \le \ell(Q) \le \ell(C) - 2$, then $\ell(Q)$ is even.
\end{lem}
\begin{proof}
Suppose to the contrary that we have a subpath $Q$ of $C$ with end vertices
long, all internal vertices short, $2\le \ell(Q) \le \ell(C) - 2$,
and $\ell(Q)$ odd.  Let $x$ and $y$ be the end vertices of $Q$.
Say $C = v_1v_2\cdots v_rv_0v_1$.  By rotating the $\alpha,\beta$ coloring of
$C$, we may assume that $x = v_0$ and $y = v_a$, where $a \ge 3$ is odd.

We now apply the \hyperref[SpecialPath]{Parallel Edge Lemma} twice, to show that $\mu(v_1v_2) \ge 2k -
d(v_0) - d(v_a)$, which contradicts that $G$ is $k$-thin.  More specifically,
assume that the edges $v_0v_1,v_1v_2,\ldots$ go clockwise around $C$.  We apply
the \hyperref[SpecialPath]{Parallel Edge Lemma} once going clockwise starting from $v_0$ and once going
counterclockwise starting from $v_a$.  The first application implies that every
color in $\vphn(v_0)$ appears on some edge parallel to $v_1v_2$; the second
implies the same for every color in $\vphn(v_a)$.  Since $\card{\vphn(v_i)}=k-d(v_i)$
for each $i\in\{0,a\}$ and $\vphn(v_0)\cap \vphn(v_a)=\emptyset$, the conclusion
follows.
\end{proof}

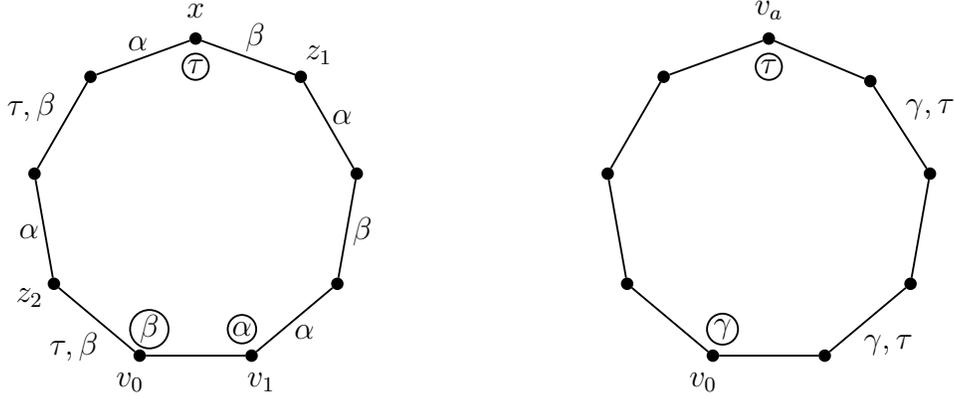
\begin{figure}
\centering
\begin{tikzpicture}[line width=.06em, rotate=90, scale=1.5]
\tikzstyle{blacknode}=[draw,circle,fill=black,minimum size=4pt,inner sep=0pt]
\tikzstyle{whitenode}=[circle,fill=white,minimum size=8pt,inner sep=0pt]
\tikzstyle{whitenode2}=[draw,circle,fill=white,minimum size=8pt,inner sep=.6pt]

\draw (20:1.5cm) node[whitenode] {$\alpha$};
\draw (40:1.45cm) node[blacknode] (b1) {};
\draw (60:1.68cm) node[whitenode] {$\tau,\beta$};
\draw (80:1.45cm) node[blacknode] (b2) {};
\draw (100:1.5cm) node[whitenode] {$\alpha$};
\draw (120:1.45cm) node[blacknode] (b3) {};
\draw (120:1.7cm) node[whitenode] {$z_2$};
\draw (140:1.68cm) node[whitenode] {$\tau, \beta$};
\draw (160:1.45cm) node[blacknode] (b4) {};
\draw (160:1.7cm) node[whitenode] {$v_0$};
\draw (160:1.20cm) node[whitenode2] {$\beta$};
\draw (200:1.45cm) node[blacknode] (b5) {};
\draw (200:1.7cm) node[whitenode] {$v_1$};
\draw (200:1.20cm) node[whitenode2] {$\alpha$};
\draw (220:1.5cm) node[whitenode] {$\alpha$};
\draw (240:1.45cm) node[blacknode] (b6) {};
\draw (260:1.5cm) node[whitenode] {$\beta$};
\draw (280:1.45cm) node[blacknode] (b7) {};
\draw (300:1.5cm) node[whitenode] {$\alpha$};
\draw (320:1.45cm) node[blacknode] (b8) {};
\draw (320:1.7cm) node[whitenode] {$z_1$};
\draw (340:1.55cm) node[whitenode] {$\beta$};
\draw (360:1.45cm) node[blacknode] (b9) {};
\draw (360:1.7cm) node[whitenode] {$x$};
\draw (360:1.20cm) node[whitenode2] {$\tau$};
\draw (b1) -- (b2) -- (b3) -- (b4) -- (b5) -- (b6) -- (b7) -- (b8) -- (b9) --
(b1);

\begin{scope}[yshift=-2in]
\draw (40:1.45cm) node[blacknode] (b1) {};
\draw (80:1.45cm) node[blacknode] (b2) {};
\draw (120:1.45cm) node[blacknode] (b3) {};
\draw (160:1.45cm) node[blacknode] (b4) {};
\draw (160:1.7cm) node[whitenode] {$v_0$};
\draw (160:1.2cm) node[whitenode2] {$\gamma$};
\draw (200:1.45cm) node[blacknode] (b5) {};
\draw (220:1.65cm) node[whitenode] {$\gamma,\tau$};
\draw (240:1.45cm) node[blacknode] (b6) {};
\draw (280:1.45cm) node[blacknode] (b7) {};
\draw (300:1.65cm) node[whitenode] {$\gamma,\tau$};
\draw (320:1.4cm) node[blacknode] (b8) {};
\draw (360:1.45cm) node[blacknode] (b9) {};
\draw (360:1.7cm) node[whitenode] {$v_a$};
\draw (360:1.2cm) node[whitenode2] {$\tau$};
\draw (b1) -- (b2) -- (b3) -- (b4) -- (b5) -- (b6) -- (b7) -- (b8) -- (b9) --
(b1);
\end{scope}
\end{tikzpicture}
\caption{The proofs of Lemmas~\ref{NonSpecialsInThinAreAtEvenDistance}
and~\ref{ThreeNonSpecialOnCycle}.\label{fig:thinLemmas}}
\end{figure}

\begin{lem}\label{ThreeNonSpecialOnCycle}
Let $G$ be a $k$-thin, critical graph with $\chi'(G) = k+1$ for some integer $k
\ge \Delta(G) + 1$.  Let $\vph$ be a $k$-edge-coloring of $G-v_0v_1$. Suppose
$\alpha \in \vphn(v_0)$ and $\beta \in \vphn(v_1)$ and let $C = P_{v_1}(\alpha,
\beta) + v_0v_1$.  If $\ell(C)\ge 5$ and $C$ contains exactly 3 long vertices,
then $C = xyAzBx$ where $A$ and $B$ are paths of even length and $x,y,z$ are
all long.  Moreover, $x$ is $y$-long and $y$ is $x$-long.
\end{lem}
\begin{proof}
Let $G$ be a graph satisfying the hypotheses, and let $x$, $y$, $z$ be the
three long vertices.
The three subpaths of $C$ with endpoints $x$,
$y$, and $z$ either (i) all have odd length or (ii) include two paths of even
length and one of odd length.  
If we are in (i), then the longest of these three subpaths
violates Lemma~\ref{NonSpecialsInThinAreAtEvenDistance}; so we are in (ii), and
also the path of odd length is simply an edge.  This proves the first statement.
For the second statement, assume to the contrary that $x$ is $y$-short.
By rotating the $\alpha,\beta$ coloring, we can assume that $y=v_0$ and $x=v_1$.
As in the previous lemma, we use the \hyperref[SpecialPath]{Parallel Edge Lemma} (and the comment in the
first paragraph of Section~\ref{sec:thin}) to conclude that $\mu(v_1v_2)\ge
2k-d(v_0)-d(z)$.  As above, this contradicts that $G$ is $k$-thin; this
contradiction proves the second statement.
\end{proof}

\begin{lem}\label{ConsecutiveNonSpecials}
Let $G$ be a non-elementary, $k$-thin, critical graph with $\chi'(G) = k+1$ for
some integer $k \ge \Delta(G) + 1$.  Choose $(T, v_0v_1, \vph) \in \T(G)$. If
$\alpha \in \vphn(v_0)$ and $\beta \in \vphn(v_1)$, then $P_{v_1}(\alpha,
\beta) + v_0v_1$ contains consecutive long vertices.
\end{lem}
\begin{proof}
Let $C = P_{v_1}(\alpha, \beta) + v_0v_1$.  By Lemma \ref{FreeColorsLemma},
there is $x \in V(C)$ and $\tau \in \vphn(x)$ such that there is a
$\tau$-colored edge from $y \in V(C)$ to $w \in V(T) \setminus V(C)$.
Lemma \ref{TauEscape} implies that $C$ has a subpath $Q$ with $x \in V(Q)$
 and long endpoints $z_1,z_2$ such that the distance from $x$ to
$z_i$ along $Q$ is odd for each $i \in \irange{2}$.  
Let $Q'$ be the subpath of $C$ with endpoints $z_1$ and $z_2$ that does not contain
$x$. Since $C$ is an odd cycle, $\ell(Q')$ is odd.  Let $Q^*$ be a minimum
length subpath of $Q'$ with long ends.  Now $\ell(Q^*) = 1$ by Lemma
\ref{NonSpecialsInThinAreAtEvenDistance}, as desired.
\end{proof}

\begin{lem}\label{MasterHelper}
Let $G$ be a non-elementary, $k$-thin, critical graph with $\chi'(G) = k+1$ for
some integer $k \ge \Delta(G) + 1$.  If $(T, v_0v_1, \vph) \in \T(G)$ and
$\nu(T) \le 3$, then $T$ contains long vertices $z_1,z_2,z_3$ such that either 
\begin{enumerate}
\item $z_1$ is $\set{z_2,z_3}$-long and $z_2$ is $z_1$-long; or
\item $z_i$ is $z_j$-long and $z_j$ is $z_i$-long for each 
$(i,j)\in\{(1,2),(2,3)\}$.
\end{enumerate}
\end{lem}
\begin{proof}
Choose $\alpha \in \vphn(v_0)$ and $\beta \in \vphn(v_1)$ so that
$P_{v_1}(\alpha, \beta)$ contains as many long vertices as possible and,
subject to that, $P_{v_1}(\alpha,\beta)$ is as long as possible.
Let $C=P_{v_1}(\alpha,\beta)+v_0v_1$.
By Lemma \ref{FreeColorsLemma}, there is $x \in V(C)$ and $\tau \in \vphn(x)$
such that there is a $\tau$-colored edge from $y \in V(C)$ to $w \in V(T)
\setminus V(C)$.  By Lemma~\ref{ConsecutiveNonSpecials}, $C$ has at least 2
long vertices.

First suppose that $C$ contains only 2 long vertices, $z_1$ and $z_2$; see the
left side of Figure~\ref{fig:MasterHelper}.  By
Lemma \ref{ConsecutiveNonSpecials}, $z_1$ and $z_2$ are consecutive on $C$.
Lemma \ref{TauEscape} implies that $C$ has a subpath $Q$ 
with endpoints $z_1,z_2$ such that $x \in V(Q)$ and $y \not \in V(Q-z_1-z_2)$ 
and for each $i \in \irange{2}$ 
there are no $\tau$-colored edges between $z_i$ and its neighbors on $C$.
By rotating the $\alpha,\beta$ coloring of $C$, we can assume that $x = v_0$ and
$\alpha,\tau\in \vphn(v_0)$ and $\beta\in\vphn(v_1)$. 
Note that $P_{v_1}(\tau,\beta)$ must end at $v_0$ (since otherwise we can
recolor the Kempe chain and color $v_0v_1$ with $\tau$).  
Let $C' = P_{v_1}(\tau, \beta) + v_0v_1$.  Note that $C'$ must include $v_1Qz_1$
and also $v_0Qz_2$ (the $\beta$-colored edges are present by definition and the
$\tau$-colored edges are present by the 
the \hyperref[SpecialPath]{Parallel Edge Lemma}).  
Thus, $z_1,z_2\in
V(C')$.  Since $z_1$ and $z_2$ are not consecutive on $C'$
and $C'$ contains no other long vertices by the maximality condition on $C$,
Lemma \ref{ConsecutiveNonSpecials} gives a contradiction.

So instead assume that $C$ contains exactly 3 long vertices. 
Now we prove that $\ell(C)\ge 5$.  Suppose, to the contrary, that $C=v_0v_1v_2$
and each vertex is a long vertex.  By Lemma~\ref{FreeColorsLemma}, some color
that is missing at a vertex of $C$ is used on an edge leaving $C$.  By
symmetry, assume that color $\tau$ is missing at $v_0$ and is used on an edge
$v_1y$, where $y\notin V(C)$.
Uncolor $v_2v_0$ and color $v_0v_1$ with $\beta$.   Now the $\beta,\tau$ path
starting at $v_0$ ends at $v_2$,  has length at least 4 and contains $v_0,
v_1, v_2$.  So the union of $v_0v_2$ with this path gives a longer $C$, a
contradiction.

By Lemma
\ref{ThreeNonSpecialOnCycle}, $C = z_1z_2Az_3Bz_1$ where $A$ and $B$ are paths
of even length.  Also, $z_1$ is $z_2$-long and $z_2$ is $z_1$-long.  
By Lemma \ref{TauEscape}, $C$ has a subpath $Q$ 
with endpoints $z_1,z_3$ and with $x \in V(Q)$ and $y \not
\in V(Q-z_1-z_3)$ 
such that there are no $\tau$-colored edges
between $z_i$ and its neighbors along $C$ for each $i \in \set{1,3}$ (it could
happen that $z_3$ has a $\tau$-colored edge parallel to an edge of $C$,
so the endpoints of $Q$ are $z_1, z_2$, but now we get a contradiction
as in the previous case, by letting $C'=P_{v_1}(\tau,\beta)+v_0v_1$).  By
rotating the $\alpha,\beta$ coloring of $C$, we
may assume that $x = v_0$.  Again, let $C' = P_{v_1}(\tau, \beta) + v_0v_1$.  We
know that $C'$ contains $z_1$ and $z_3$ and that $z_1$ and $z_2$ are not
consecutive on $C'$.  Note also that all long vertices in $V(C')$ must be among
$z_1,z_2,z_3$, since otherwise $\nu(T)\ge 4$, contradicting our hypothesis.
So by Lemma~\ref{ConsecutiveNonSpecials}, either $z_1$ and
$z_3$ are consecutive on $C'$ or $z_2$ and $z_3$ are consecutive on $C'$.

Suppose that $z_2$ and $z_3$ are consecutive on $C'$, and thus connected by a
$\tau$-colored edge.  Now applying Lemma \ref{ThreeNonSpecialOnCycle} shows that $z_2$
is $z_3$-long and $z_3$ is $z_2$-long, so we satisfy (2) in the conclusion of
the lemma (by swapping the names of $z_1$ and $z_2$).

So instead $z_1$ and $z_3$ must be consecutive on $C'$, and thus connected by a
$\tau$-colored edge.  If $z_1 = v_1$, then we have a fan with an
$\alpha$-colored edge from $z_1$ to $z_2$ and a $\tau$-colored edge from $z_1$
to $z_3$, so $z_1$ is $\set{z_2,z_3}$-long. 

Now assume that $z_1\ne v_1$; see the right side of
Figure~\ref{fig:MasterHelper}.
Let $z_1'$ be the predecessor of $z_1$ on the path from $v_0$ (through $v_1$) to
$z_1$.  We can shift the coloring so that $z_1'z_1$ is uncolored and $z_1z_2$
is colored $\alpha$ (as in the proof of the \hyperref[SpecialPath]{Parallel
Edge Lemma}).  In fact, 
we can shift either the $\alpha,\beta$ edges or the $\tau,\beta$ edges.  This
gives the options that either $\alpha\in \vphn(z_1')$ or $\tau\in \vphn(z_1')$,
whichever we prefer.  Suppose we shift the $\tau,\beta$ edges.
Now choose $\gamma\in \vphn(z_1')-\alpha-\tau$.  Consider
the $\gamma$-colored edge $e$ incident to $z_1$.  If $e$ goes to $z_2$, then we
$z_1$ is $\{z_2,z_3\}$-long, by colors $\gamma$ and $\tau$; so we satisfy (1) in
the conclusion of the lemma.
If instead $e$ goes to $z_3$, then instead of shifting the $\tau,\beta$ edges we
shift the $\alpha,\beta$ edges; note that this recoloring preserves the fact
that $\gamma$ is missing at $z_1'$.  Now again $z_1$ is $\{z_2,z_3\}$-long, this
time by colors $\alpha$ and $\gamma$; so we again satisfy (1) in the conclusion
of the lemma.

Finally, assume that the $\gamma$-colored edge incident to $z_1$ goes to some
vertex other than $z_2$ and $z_3$.  Now let $C''=P_{z_1}(\gamma,\beta)+z_1z_1'$.
Since $V(C'')\subseteq V(T)$, Lemmas~\ref{ConsecutiveNonSpecials} and
\ref{ThreeNonSpecialOnCycle} imply that $z_2$ and $z_3$ are
adjacent on $C''$ and furthermore $z_2$ is $z_3$-long and $z_3$ is $z_2$-long;
thus, we satisfy (2) in the conclusion of the lemma.
\end{proof}

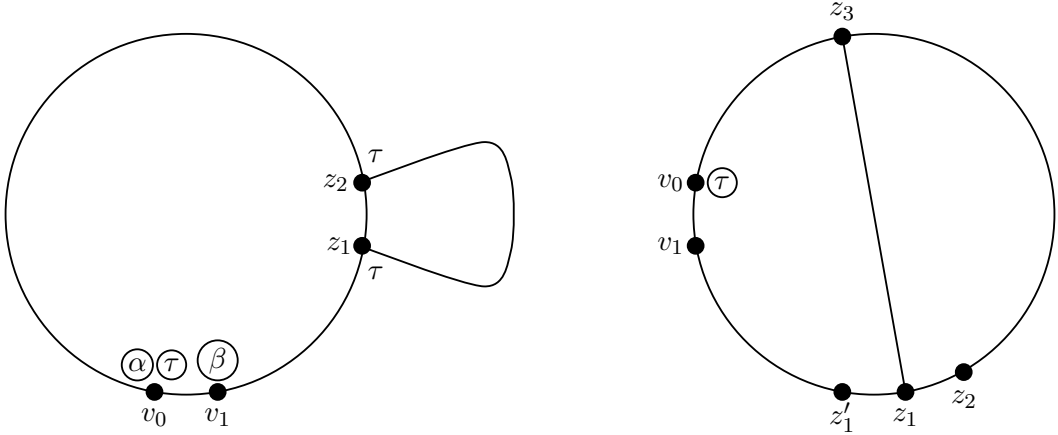
\begin{figure}
\centering
\begin{tikzpicture}[scale = 12, line width=.06em]
\tikzstyle{VertexStyle} = []
\tikzstyle{EdgeStyle} = []
\tikzstyle{edgeLeft} = [style={bend left=10}]
\tikzstyle{edgeRight} = [style={bend right=10}]
\tikzstyle{labeledStyle}=[shape = circle, minimum size = 6pt, inner sep = 1.2pt]
\tikzstyle{labeledStyle2}=[shape = circle, minimum size = 6pt, inner sep =
1.2pt, draw]
\tikzstyle{unlabeledStyle}=[shape = circle, minimum size = 6pt, inner sep = 1.2pt, draw, fill]
\draw (0,.5) circle (.20cm);
\Vertex[style = labeledStyle2, x = -0.054, y = 0.333, L = \small {$\alpha$}]{v00}
\Vertex[style = labeledStyle2, x = -0.016, y = 0.333, L = \small {$\tau$}]{v0}
\Vertex[style = labeledStyle2, x = 0.035, y = 0.338, L = \small {$\beta$}]{v1}
\Vertex[style = unlabeledStyle, x = -.035, y = 0.303, L = \tiny {}]{v2}
\Vertex[style = labeledStyle, x = -0.035, y = 0.275, L = \small {$v_0$}]{v3}
\Vertex[style = unlabeledStyle, x = 0.035, y = 0.303, L = \tiny {}]{v4}
\Vertex[style = labeledStyle, x = 0.035, y = 0.275, L = \small {$v_1$}]{v5}
\draw[] plot [smooth] coordinates {(.1925,.465) (.33,.42)
(.36,.46) (.36,.54) (.33,.58) (.1925,.535)}; 
\Vertex[style = labeledStyle, x = 0.21, y = 0.435, L = \small {$\tau$}]{v10}
\Vertex[style = labeledStyle, x = 0.21, y = 0.565, L = \small {$\tau$}]{v11}
\Vertex[style = unlabeledStyle, x = 0.195, y = 0.465, L = \small {}]{v12}
\Vertex[style = unlabeledStyle, x = 0.195, y = 0.535, L = \small {}]{v13}
\Vertex[style = labeledStyle, x = 0.17, y = 0.465, L = \small {$z_1$}]{v14}
\Vertex[style = labeledStyle, x = 0.166, y = 0.535, L = \small {$z_2$}]{v15}


\begin{scope}[xshift=.3in]
\draw (0,.5) circle (.20cm);
\Vertex[style = unlabeledStyle, x = -.035, y = 0.303, L = \tiny {}]{v2}
\Vertex[style = labeledStyle, x = -0.035, y = 0.275, L = \small {$z'_1$}]{v3}
\Vertex[style = unlabeledStyle, x = 0.035, y = 0.303, L = \tiny {}]{v4}
\Vertex[style = labeledStyle, x = 0.035, y = 0.275, L = \small {$z_1$}]{v5}
\Vertex[style = unlabeledStyle, x = 0.0995, y = 0.325, L = \small {}]{v6}
\Vertex[style = labeledStyle, x = 0.10000, y = 0.295, L = \small {$z_2$}]{v7}
\Vertex[style = unlabeledStyle, x = -0.1975, y = 0.535, L = \small {}]{v16}
\Vertex[style = labeledStyle2, x = -0.168, y = 0.535, L = \small {$\tau$}]{v19}
\Vertex[style = unlabeledStyle, x = -0.1975, y = 0.465, L = \small {}]{v16}
\Vertex[style = labeledStyle, x = -0.225, y = 0.535, L = \small {$v_0$}]{v18}
\Vertex[style = labeledStyle, x = -0.225, y = 0.465, L = \small {$v_1$}]{v28}
\Vertex[style = unlabeledStyle, x = -.035, y = 0.697, L = \tiny {}]{v21}
\Vertex[style = labeledStyle, x = -0.035, y = 0.725, L = \small {$z_3$}]{v20}
\draw (-.035, .697) -- (.035, .303);
\end{scope}
\end{tikzpicture}
\caption{Two parts of the proof of Lemma~\ref{MasterHelper}.\label{fig:MasterHelper}}
\end{figure}

We need the following result from~\cite{rabern2011strengthening}, which we use
to handle the case when $G$ is not $k$-thin.

\begin{thm}[\cite{rabern2011strengthening}]\label{CriticalMuBound}
If $Q$ is the line graph of a graph $G$ and $Q$ is vertex critical, then
\[\chi(Q) \leq \max\left\{\omega(Q), \Delta(Q) + 1 - \frac{\mu(G) - 1}{2}\right\}.\]
\end{thm}

Now we prove the main result of this section, the Weak $\frac56$-Theorem.  It
encapsulates most of what we will need from the first four sections when we prove
our main result, the $\frac56$-Theorem, in Section~\ref{sec:final}.

\begin{thm}[Weak $\frac56$-Theorem]
If $Q$ is the line graph of $G$, then
\[\chi(Q) \le \max\set{\W(G), \Delta(G) + 1, \frac{5\Delta(Q) + 8}{6}}.\]
\label{mainhelper}
\end{thm}
\begin{proof}
Suppose the theorem is false and choose a counterexample minimizing $\card{Q}$.
Let $k = \max\set{\W(G), \Delta(G) + 1, \floor{\frac{5\Delta(Q) + 8}{6}}}$. Say $Q =
L(G)$ for a graph $G$. The minimality of $Q$ implies that
$G$ is critical and $\chi(Q) = k+1$, for some $k \ge \Delta(G) + 1$.

The heart of the proof is Claim~1, which roughly says that if $x$ is long, then
$d(x)<\frac34\Delta(G)$. Moreover, we can improve this bound further if $x$ is
the root of a long fan $F$ such that either (i) $F$ has length more than 3 or (ii)
some of the other vertices in $F$ have degree less than $\Delta(G)$.  The claims
thereafter are all essentially applications of Claim~1.
\bigskip

\claim{1}{Let $F$ be a fan rooted at $x$ with respect to a $k$-edge-coloring of
$G - xy$.  If $S\subseteq V(F)-x$ and $\card{S} \ge 3$, then
\[d(x) \le \frac1{5\card{S}-11}\parens{2\card{S}-12 + \sum_{v \in S} d(v)}.\]
In particular, if $\card{S}=3$, then $d(x) \le \frac1{4}\parens{-6 + \sum_{v \in S}
d(v)}.$}
\begin{claimproof}
Since $F$ is elementary, we have
\[2 + k-d(x) + \sum_{v \in S} k - d(v) \le k,\]
so
\[2 + \card{S}k \le d(x) + \sum_{v \in S} d(v).\]
Using $k \ge \frac56(\Delta(Q) + 1) - \frac13 \ge \frac56(d(x) + d(v) - \mu(xv))
- \frac13$ for each $v
\in S$, we get
\[2 + \sum_{v \in S}\parens{\frac56(d(x) + d(v) - \mu(xv)) -\frac13} \le d(x) +
\sum_{v \in S} d(v),\]
so multiplying by 6 and rearranging terms gives
\[12 + \parens{5\card{S} - 6}d(x) - 2\card{S} \le \sum_{v \in S} 5\mu(xv) + \sum_{v \in S} d(v).\]
Now $\sum_{v \in S} \mu(xv) \le d(x)$, so this implies
\[12 + \parens{5\card{S} - 11}d(x) - 2\card{S} \le \sum_{v \in S} d(v).\]
Solving for $d(x)$ gives
\[d(x) \le \frac1{5\card{S}-11}\parens{2\card{S}-12 + \sum_{v \in S} d(v)},\]
and when $\card{S} = 3$, we get
$d(x) \le \frac14\parens{-6 + \sum_{v \in S} d(v)}.$
\end{claimproof}
\bigskip

\claim{2}{If $x \in V(G)$ is long, then $d(x) \le \frac34\Delta(G) - 1$.}

\begin{claimproof}
This is immediate from Claim 1, since $d(v)\le \Delta(G)$ for all $v\in S$.
\end{claimproof}
\bigskip

\claim{3}{If $x_1x_2 \in E(G)$ such that $x_1$ is $x_2$-long and $x_2$ is
$x_1$-long, 
then
\[d(x_i) \le \frac23\Delta(G) -2 \text{ for all $i \in \irange{2}$.}\]}

\begin{claimproof}
By Claim 1, for each $i \in \irange{2}$,
\[d(x_i) \le \frac14\parens{-6 + \sum_{v \in S} d(v)} \le \frac14\parens{-6 + d(x_{3-i}) + 2\Delta(G)},\]
Substituting the bound on $d(x_{3-i})$ into that on $d(x_i)$ and simplifying
gives for each $i \in \irange{2}$,
\[d(x_i) \le -2 + \frac23\Delta(G).\]
\end{claimproof}

\bigskip

\claim{4}{If $x_1x_2, x_1x_3 \in E(G)$ such that $x_1$ is $\set{x_2,x_3}$-long, $x_2$ is $x_1$-long and $x_3$ is long, then 
\[d(x_1) \le -\frac85 + \frac35\Delta(G),\]
\[d(x_2) \le -\frac75 + \frac{13}{20}\Delta(G).\]}

\begin{claimproof}
By Claim 1, we have
\[d(x_1) \le \frac14\parens{-6 + \sum_{v \in S} d(v)} \le \frac14\parens{-6 + d(x_2) + d(x_3) + \Delta(G)},\]
\[d(x_2) \le \frac14\parens{-6 + \sum_{v \in S} d(v)} \le \frac14\parens{-6 + d(x_1) + 2\Delta(G)}.\]
By the same calculation as in Claim~3, these together imply
\[d(x_1) \le -2 + \frac25\Delta(G) + \frac{4}{15}d(x_3).\]
Since $x_3$ is long, using Claim 2, we get
\[d(x_1) \le -\frac{34}{15} + \frac35\Delta(G),\]
and hence
\[d(x_2) \le -\frac{61}{15} + \frac{13}{20}\Delta(G).\]
\end{claimproof}
\bigskip

\claim{5}{The theorem is true.}

\begin{claimproof}
Let $(T, v_0v_1, \vph) \in \T(G)$. By Lemma \ref{MasterHelper}, one of the following holds:
\begin{enumerate}
\item $G$ is elementary; or
\item $G$ is not $k$-thin; or
\item $\nu(T) = 3$ and $V(T)$ contains vertices $x_1,x_2,x_3$ such that $x_1$
is $x_2$-long, $x_2$ is $x_1$-long, $x_2$ is $x_3$-long, and $x_3$ is $x_2$-long; or
\item $\nu(T) = 3$ and $V(T)$ contains vertices $x_1,x_2,x_3$ such that $x_1$ is
$\set{x_2,x_3}$-long, $x_2$ is $x_1$-long, and $x_3$ is long; or
\item $V(T)$ contains four long vertices $x_1, x_2, x_3, x_4$.
\end{enumerate}

If (1) holds, then $\chi(Q) = \W(G)$, which contradicts our choice of
$Q$ as a counterexample.

If (2) holds, then Claim 2 implies that $\mu(G) \ge 2k - \frac32\Delta(G) + 2$. 
Now Theorem \ref{CriticalMuBound} gives
\begin{align*}
k + 1 &\le \Delta(Q)+1-\frac{2k-\frac32\Delta(G)+2}2\\
&=\Delta(Q) + 1 - k +
\frac34\Delta(G) - 1,
\end{align*}
so
\[2(k + 1) \le \Delta(Q) + 1 + \frac34\Delta(G).\]

Substituting $\Delta(G) \le k$ and solving for $k$ gives

\[k  \le \frac45\Delta(Q) - \frac45 < \frac56\Delta(Q)+\frac12 \le k,\]
which is a contradiction.

Suppose (3) holds.  
Now \[2 + \sum_{i \in \irange{3}} k - d(x_i) \le k,\]
so Claim 3 implies
\[3\parens{\frac23\Delta(G) -2} \ge 2k+2,\]
which is a contradiction, since $\Delta(G) \le k$.

Suppose (4) holds.  Now
\[2 + \sum_{i \in \irange{3}} k - d(x_i) \le k,\]
so Claims 2 and 4 give
\[ \parens{\frac35 + \frac{13}{20} + \frac34}\Delta(G)-\parens{\frac{34}{15} + \frac{16}{15} + 1}\ge 2k+2,\]
which is
\[2\Delta(G) -\frac{13}3\ge 2k+2,\]
again a contradiction, since $\Delta(G) \le k$.

So (5) must hold.  But now
\[2 + \sum_{i \in \irange{4}} k - d(x_i) \le k,\]
so using Claim 2 gives
\[4\parens{\frac34\Delta(G) - 1} \ge 3k+2,\]
a contradiction since $\Delta(G) \le k$.
\end{claimproof}

This finishes the final case of Claim~5, which proves the theorem.
\end{proof}

In the previous theorem, we showed that 
$\chi(Q) \le \max\set{\W(Q), \Delta(G) + 1, \frac{5\Delta(Q) +
8}{6}}$.  Now we show that if the maximum is attained by the second argument
(and $Q$ is vertex critical), then $Q$ satisfies the $\frac56$-Conjecture.
We use the following lemma, which is implicit in~\cite{rabern2011strengthening};
see the proof of Lemma~9 therein.

\begin{lem}\label{CriticalMuBoundOtherWay}
If $Q$ is the line graph of a graph $G$ and $Q$ is vertex critical, then
\[\chi(Q) \leq \max\left\{\Delta(G), \Delta(Q) + 1 + 2\mu(G) - \Delta(G)\right\}.\]
\end{lem}

\begin{cor}
If $Q$ is the line graph of a critical graph $G$ and $\chi(Q) \le \Delta(G) + 1$, then
\[\chi(Q) \le \max\set{\omega(Q), \frac{5\Delta(Q) + 8}{6}}.\]
\label{mainCorHelper}
\end{cor}
\begin{proof}
Let $k +1=\chi(Q) \le \Delta(G)+1$.  Suppose $\chi(Q) > \omega(Q)$.  Now Lemma
\ref{CriticalMuBoundOtherWay} gives
\[k + 1 = \chi(Q) \le \Delta(Q) + 1 + 2\mu(G) - k,\]
so solving for $\mu(G)$ gives
\[\mu(G) \ge k - \frac{\Delta(Q)}{2}.\]
Applying Theorem \ref{CriticalMuBound} gives
\[k+1 = \chi(Q) \le \Delta(Q) + 1 - \frac{k - \frac{\Delta(Q)}{2} - 1}{2},\]
and solving for $k+1$ yields
\[\chi(Q) = k+1 \le \frac56\Delta(Q) + \frac43 = \frac{5\Delta(Q) + 8}{6}.\]
\aftermath
\end{proof}

Since $\omega(Q) \le \max\{\Delta(G),\W(G)\}$, Theorem~\ref{mainhelper} and
Corollary~\ref{mainCorHelper} together imply the following.

\begin{cor}
If $Q$ is the line graph of a graph $G$, then
\[\chi(Q) \le \max\set{\Delta(G),\W(G), \frac{5\Delta(Q) + 8}{6}}.\]
\label{mainCor}
\end{cor}
\begin{proof}
Let $Q$ be the line graph of a graph $G$.  
We assume that $G$ is critical.  
If not, then choose $\widehat{G}\subseteq G$ such that $\widehat{G}$ is critical
and $\chi'(\widehat{G})=\chi'(G)$.  Let $\widehat{Q}\DefinedAs L(\widehat{G})$. 
Now $\chi(Q)=\chi(\widehat{Q}) \le
\max\{\Delta(\widehat{Q}),\W(\widehat{Q}),\frac{5\Delta(\widehat{Q})+8}6\}
\le\max\{\Delta(Q),\W(Q),\frac{5\Delta(Q)+8}6\}$, as desired.

If $\chi(Q) >\Delta(G)+1$, then Theorem~\ref{mainhelper} implies that
$\chi(Q)\le\max\{\W(G),\frac{5\Delta(Q)+8}6\}$.  Otherwise, $\chi(Q)\le
\Delta(G)+1$; since $G$ is critical, Corollary~\ref{mainCorHelper} implies 
$\chi(Q)\le\{\omega(Q),\frac{5\Delta(Q)+8}6\}$.  Since
$\omega(G)\le\max\{\Delta(G),\W(G)\}$, the result follows.
\end{proof}

\section{The $\frac56$-Conjecture}
\label{sec:final}

In this section, we prove our main result, Theorem~\ref{thm:main}, 
that $\chi(Q)\le\max\{\omega(Q),\frac{5\Delta(Q)+8}6\}$, when $Q$ is the line
graph of a graph $G$.
Throughout, we may assume that $Q$ is a minimal counterexample, so $Q$ is vertex
critical.
Observe that if $\chi(Q)\le \Delta(G)+1$, then the result follows immediately
from Corollary~\ref{mainCorHelper}.  Thus, we may assume that
$\chi(Q)>\Delta(G)+1$.  Now, in view of Corollary~\ref{mainCor}, it suffices to
show that $\chi(Q)=\chi'(G)\le \W(G)$ implies
$\chi(Q)\le\max\{\omega(Q),\frac{5\Delta(Q)+8}6\}$.

Our approach is to show that if $Q$ is a minimal counterexample and $Q$ is the
line graph of $G$, then $\card{N(x)}=2$ for nearly every vertex $x\in V(G)$.  This
implies that the simple graph underlying $G$ is very close to a cycle.
By Corollary~\ref{mainCor}, we can assume that $G$ is elementary, i.e.,
$\chi(Q)=\chi'(G)=W(G)$.  Thus
$\chi(Q)=\ceil{\frac{2\card{E(G)}}{\card{V(G)}-1}}\ge\Delta(G)+2$.
So, to show that $\card{N(x)}=2$ for nearly every $x\in V(G)$, it suffices to show
that when $\card{N(x)}\ge 3$, we have $d(x)\le \frac35\Delta(G)$.  This is an
immediate consequence of Lemma~\ref{DegreeBoundedForMiddling}, which is the key
technical result that we prove in this section.

It is helpful to note that in the present section the only results we use from
previous sections are Corollary~\ref{mainCorHelper} and Corollary~\ref{mainCor}
(as well as Theorem~\ref{CriticalMuBound}, which is proved
in~\cite{rabern2011strengthening}).
In Lemmas~\ref{Slacked}--\ref{DegreeBoundedForMiddling},
we prove bounds on $\card{N(x)}$ for each $x\in V(G)$.  Ultimately, these lemmas
yield Corollary~\ref{CriticalElementary56}, which says that $\card{N(x)}\ge 3$
for at most one vertex $x\in V(G)$.  This corollary plays a key role in the
proof of our main result, Theorem~\ref{thm:main}.

For reference, we record the following proposition, which is Proposition~1.3
in~\cite{SSTF}.
\begin{prop}
\label{easy-prop}
Let $G$ be a graph that is critical and elementary, with $\chi'(G)=k+1$ for some
integer $k\ge \Delta(G)$.  Now $\size{G}$ is odd and
$k=\frac{2(\size{G}-1)}{\card{G}-1}$.
\end{prop}

\begin{lem}
\label{Slacked}
Let $G$ be a critical, elementary graph with $\chi'(G) = k + 1$ where $k \ge \Delta(G) + 1$.  Put $Q \DefinedAs L(G)$.  
If $k = \epsilon\parens{\Delta(Q) + 1} + \beta$, then for all $x \in V(G)$,
\[\card{N(x)} = \frac{\epsilon\parens{\card{G} - \Delta(G) - d_G(x) - 1 + S_1 + S_2 + S_3\parens{\card{G} - 1}}}{(1-\epsilon)\Delta(G) - \epsilon d_G(x) + 1 - \beta + S_3},\]
where 
\[S_1 \DefinedAs \sum_{v \in N(x)} \Delta(Q) - d_Q(xv),\]
\[S_2 \DefinedAs 2 + \sum_{v \in V(G) \setminus N(x)} \Delta(G) - d_G(v),\]
\[S_3 \DefinedAs k - (\Delta(G) + 1).\]
\end{lem}
\begin{proof}
The details of the proof are somewhat tedious, however, the idea is simple.  We
write down a system of equations and solve for $\card{N(x)}$.  The main
insight needed is knowing which quantities to consider.  We use $\epsilon$,
$\beta$, $S_1$, $S_2$, $S_3$, and $\sum_{v\in N(x)}d_G(v)$.

Since $G$ is critical and elementary, Proposition~\ref{easy-prop} implies that
$\card{G}$ is odd and
\begin{equation}
\label{eq1}
k = \frac{2(\size{G} - 1)}{\card{G} - 1}.
\end{equation}
Choose $x \in V(G)$ and put $M \DefinedAs \card{N(x)}$ and 
\[P \DefinedAs \sum_{v \in N(x)} d_G(v).\] 
Grouping edges by whether or not they are incident to any neighbor of $x$ gives
\begin{equation}\label{eq2}
2(\size{G} - 1) = \Delta(G)(\card{G} - M) - S_2 + P.
\end{equation}
By~\eqref{eq1}, and by the definition of $S_3$,
\[\frac{2(\size{G} - 1)}{\card{G} - 1} = k = \Delta(G) + 1 + S_3.\]
Now clearing the denominator and using \eqref{eq2} to substitute, we get
\[P = (\card{G} - 1)(\Delta(G) + 1 + S_3) - \Delta(G)(\card{G} - M) + S_2.\]
After regrouping terms, this is
\begin{equation}\label{eq3}
P = \Delta(G)(M-1) + \card{G} - 1 + S_2 + S_3(\card{G} - 1).
\end{equation}
Using $k = \epsilon\parens{\Delta(Q) + 1} + \beta$, and the definition of $S_1$, we get
\[kM = \beta M + \epsilon S_1 + \epsilon\sum_{v \in N(x)} d_G(x) + d_G(v) - \mu(xv).\]
Since $\sum_{v \in N(x)} \mu(xv) = d_G(x)$, we have
\begin{equation}\label{eq4}
kM = \beta M + \epsilon S_1 + \epsilon d_G(x)(M - 1) + \epsilon P.
\end{equation}
Substituting \eqref{eq3} into \eqref{eq4} and solving for $M$ gives
\[M= \frac{\epsilon\parens{\card{G} - \Delta(G) - d_G(x) - 1 + S_1 + S_2 + S_3\parens{\card{G} - 1}}}{(1-\epsilon)\Delta(G) - \epsilon d_G(x) + 1 - \beta + S_3},\]
as desired.
\end{proof}

Using $\epsilon = \frac56$ in Lemma~\ref{Slacked}, we get the following.

\begin{lem}
\label{Slacked56}
Let $G$ be a critical, elementary graph with $\chi'(G) = k + 1$ where $k \ge \Delta(G) + 1$.  Put $Q \DefinedAs L(G)$. 
If $k = \frac56\parens{\Delta(Q) + 1} + \beta$, then for all $x \in V(G)$,
\[\card{N(x)} = \frac{5\parens{\card{G} - \Delta(G) - d_G(x) - 1 + S_1 + S_2 + S_3\parens{\card{G} - 1}}}{\Delta(G) - 5 d_G(x) + 6(1 - \beta + S_3)},\]
where 
\[S_1 \DefinedAs \sum_{v \in N(x)} \Delta(Q) - d_Q(xv),\]
\[S_2 \DefinedAs 2 + \sum_{v \in V(G) \setminus N(x)} \Delta(G) - d_G(v),\]
\[S_3 \DefinedAs k - (\Delta(G) + 1).\]
\end{lem}

\begin{lem}
\label{DegreeBoundedForMiddling}
Let $G$ be a critical, elementary graph with $\chi'(G) = k + 1$, where $k \ge
\Delta(G) + 1$.  Put $Q \DefinedAs L(G)$.  If $k = \frac56\parens{\Delta(Q) +
1} + \beta$, where $\beta \ge -\frac13$, then for all $x \in V(G)$ with
$\card{N(x)} \ge 3$, \[d_G(x) \le \frac{3}{5}\Delta(G) - \frac{1}{\card{N(x)} -
2}\sum_{v \in V(G)\setminus N[x]} \Delta(G) - d_G(v).\]
If additionally, $\card{N(x)} \le \frac58\card{G}$, then
\[d_G(x) \le \frac{\card{N(x)}}{5\parens{\card{N(x)} - 2}}\Delta(G) - \frac{1}{\card{N(x)} - 2}\sum_{v \in V(G)\setminus N[x]} \Delta(G) - d_G(v).\]
\end{lem}
\begin{proof}
We may assume that $\card{G}\ge 5$, since $\card{G}$ is odd by
Proposition~\ref{easy-prop}, and if $\card{G}=3$, then $\card{N(x)}\le 2$ for
all $x\in V(G)$, so there is nothing to prove.

Choose $x\in V(G)$ with $\card{N(x)}\ge 3$.
Let $S_4=\card{N(x)}-2$, and note that $S_4\ge 1$.
Applying Lemma \ref{Slacked56} and simplifying using $S_1 \ge 0$ and $\beta \ge -\frac13$ gives
\begin{equation}\label{longeq}
(5+5S_4)d_G(x) \le (7 + S_4)\Delta(G) - 5\card{G} + 21 + S_3(-5\card{G} + 17 + 6S_4) + 8S_4 - 5S_2.
\end{equation}
Put 
\[t \DefinedAs \sum_{v \in V(G) \setminus N[x]} \Delta(G) - d_G(v).\]
Now $S_2 = t + 2 + \Delta(G) - d_G(x)$.  Using this in \eqref{longeq}, we get
\begin{equation}\label{longeq2}
5S_4d_G(x) \le (2 + S_4)\Delta(G) - 5\card{G} + 11 + S_3(-5\card{G} + 17 + 6S_4) + 8S_4 - 5t.
\end{equation}
When $\card{N(x)}\le\frac58\card{G}$, i.e.,
$S_4 \le \frac58\card{G} - 2$, we have
\[- 5\card{G} + 11 + S_3(-5\card{G} + 17 + 6S_4) + 8S_4 \le 0,\]
so dividing \eqref{longeq2} through by $5S_4$ gives the desired bound.

So instead assume $S_4 > \frac58\card{G} - 2$. Rearranging \eqref{longeq2} gives 
\begin{equation}\label{longeq3}
5S_4d_G(x) \le 3S_4\Delta(G) - (2S_4-2)\Delta(G) - 5\card{G} + 11 +
S_3(-5\card{G} + 17 + 6S_4) + 8S_4 - 5t
\end{equation}
Now $S_4 = \card{N(x)} - 2 = \card{G} - 3 + S_5$, where
\[S_5 \DefinedAs S_4 + 3 - \card{G} \le 0.\]
Thus
\[-5\card{G} + 15 + 5S_4 +S_3(-5\card{G} + 15 + 5S_4) - 5S_5(1 + S_3) = 0.\]
Subtracting this equality from \eqref{longeq3} gives
\begin{equation}\label{longeq4}
5S_4d_G(x) \le 3S_4\Delta(G) - (2S_4-2)\Delta(G) - 4 + 2S_3 + (S_3 + 3)S_4 + 5S_5(1 + S_3) - 5t 
\end{equation}
If
\[- (2S_4-2)\Delta(G) -4 + 2S_3 + (S_3 + 3)S_4 + 5S_5(1 + S_3) \le 0,\]
then we have the desired bound
\[d_G(x) \le \frac{3}{5}\Delta(G) - \frac{1}{\card{N(x)} - 2}\sum_{v \in V(G)\setminus N[x]} \Delta(G) - d_G(v).\]
So assume instead that
\[-4 + 2S_3 + (S_3 + 3)S_4 - (2S_4-2)\Delta(G) + 5S_5(1 + S_3) > 0,\]
which we rewrite as
\begin{equation}\label{MM3}
(2 + S_4)S_3 + 3S_4 > (2S_4-2)\Delta(G) + 4 - 5S_5(S_3 + 1).
\end{equation}
By Shannon's theorem $k + 1 \le \frac32\Delta(G)$, so $S_3 \le \frac{\Delta(G)}{2} - 2$. After plugging in for $S_3$ on the left side and solving for $S_4$, we get
\[S_5 + \card{G} - 3 = S_4 < \frac{6\Delta(G) - 16 + 10S_5(S_3+1)}{3\Delta(G) -
2} = 2 + \frac{10S_5(S_3 + 1) - 12}{3\Delta(G) - 2},\]
so
\[\card{G} < 5 + \frac{(10S_3 - 3\Delta(G) + 12)S_5 - 12}{3\Delta(G) - 2}.\]
Since $S_5 \le 0$, this implies $\card{G} \le 3$, unless $10S_3 - 3\Delta(G) + 12 < 0$.  So $S_3 < \frac{3}{10}\Delta(G) - \frac65$. 
Since $S_5 \le 0$, \eqref{MM3} implies
\[(2 + S_4)S_3 + 3S_4 > (2S_4-2)\Delta(G) + 4\]
Substituting $S_3 < \frac{3}{10}\Delta(G)-\frac65$ gives
\[\card{G} - 3 \le S_4 < \frac{26\Delta(G) - 64}{17\Delta(G) - 18} < 2,\]
which contradicts that $\card{G}\ge 5$.
\end{proof}

\begin{cor}
\label{CriticalElementary56}
Let $G$ be a critical, elementary graph with $\chi'(G) = k + 1$, where $k \ge
\Delta(G) + 1$.  Put $Q \DefinedAs L(G)$.  If $k = \frac56\parens{\Delta(Q) +
1} + \beta$, where $\beta \ge -\frac13$, then there is at most one $x \in V(G)$
with $\card{N(x)} \ge 3$.
\end{cor}
\begin{proof}
Since $G$ is critical and elementary, Proposition~\ref{easy-prop} implies that
$\card{G}$ is odd and
\[\frac{2(\size{G} - 1)}{\card{G} - 1} = k \ge \Delta(G) + 1,\]
so
\[2\size{G} \ge \Delta(G)\card{G} + \card{G} - \Delta(G) + 1.\]
In particular,
\begin{align}
\sum_{v \in V(G)} \Delta(G) - d_G(v) &\le \Delta(G) - 1 - \card{G}.
\label{eqB1}
\end{align}
By Lemma \ref{DegreeBoundedForMiddling}, every $x \in V(G)$ with $\card{N(x)} \ge 3$ has $d_G(x) \le \frac35\Delta(G)$, so there are at most two such $x$ since
$\frac25 + \frac25 + \frac25 > 1$.  Suppose there are $x_1, x_2$ with $\card{N(x_1)} \ge \card{N(x_2)} \ge 3$.

Choose $z \in V(G)$ with $d_G(z) = \Delta(G)$.  By Lemma
\ref{DegreeBoundedForMiddling}, $\card{N(z)} = 2$, so $\mu(G) \ge \frac12\Delta(G)$.
By Theorem \ref{CriticalMuBound}, $\mu(G) < \frac13(\Delta(Q) + 1)$.  We conclude
\begin{equation}
\Delta(G) < \frac23(\Delta(Q) + 1).\label{DeltaGQUpperBound}
\end{equation}

First, suppose $x_1$ and $x_2$ are adjacent.  Since $Q$ is vertex-critical, 
\begin{align*}
k &\le d_Q(x_1x_2)\\
&=d_G(x_1) + d_G(x_2) - \mu(x_1x_2) - 1\\
&\le \frac65\Delta(G) - \mu(x_1x_2) - 1,
\end{align*}
So, $\Delta(G) >\frac56k$.  By \eqref{DeltaGQUpperBound},
\[\frac56k < \Delta(G) < \frac23(\Delta(Q) + 1),\]
and hence $k < \frac45(\Delta(Q) + 1)$, a contradiction.  

So instead assume $x_1$ and $x_2$ are non-adjacent. Suppose also that
$\card{N(x_2)} \le \frac58\card{G}$.  If $\card{N(x_2)}\ge 4$, then 
$d_G(x_2)\le \frac25\Delta(G)$.
Now~\eqref{eqB1} and Lemma~\ref{DegreeBoundedForMiddling} give
\begin{align*}
\Delta(G)-1-\card{G}&\ge \sum_{v\in V(G)}\Delta(G)-d_G(v)\\
&\ge (\Delta(G)-d_G(x_1))+(\Delta(G)-d_G(x_2))\\
&\ge \left(\frac25+\frac35\right)\Delta(G)=\Delta(G),
\end{align*}
a contradiction.

So assume that $\card{N(x_2)}=3$.
Since $x_1$ and $x_2$ are non-adjacent, 
Lemma~\ref{DegreeBoundedForMiddling} gives \[d_G(x_2) \le \frac35\Delta(G) -
(\Delta(G) - d_G(x_1)) \le \frac15\Delta(G).\] 
Similar to above, we get a contradiction since $\Delta(G)>\left(\frac45 +
\frac25\right)\Delta(G) > \Delta(G)$. 

So assume that $\card{N(x_i)} > \frac58\card{G}$ for each $i \in \irange{2}$. 
In particular, there is $y \in N(x_1) \cap N(x_2)$.  Since $\card{N(y)} = 2$,
by symmetry we may assume $\mu(x_1y) \ge \frac12 d_G(y)$.  Hence, using
\begin{align*}
k &\le d_Q(x_1y) \\
&= d_G(x_1) + d_G(y) - \mu(x_1y) < d_G(x_1) + \frac12d_G(y) \\
&\le \frac35\Delta(G)+\frac12\Delta(G)=\frac{11}{10}\Delta(G) \\
&< \frac{11}{15}(\Delta(Q) + 1),
\end{align*}
where the final inequality holds by \eqref{DeltaGQUpperBound}.
This contradiction completes the proof.
\end{proof}

Now we prove the Main Theorem.

\begin{thm}[$\frac56$-Theorem]
If $Q$ is a line graph, then
\[\chi(Q)\le \max\set{\omega(Q),\frac{5\Delta(Q)+8}{6}}.\]
\label{thm:main}
\end{thm}
\begin{proof}
It is convenient to note that, since $\chi(Q)$ is an integer,
$\chi(Q)\le \frac{5\Delta(Q)+8}{6}$ if and only if
$\chi(Q)\le \ceil{\frac{5\Delta(Q)+3}{6}}.$  
For the present proof, it is simpler to work with the latter formulation.

Suppose the theorem is false and choose a counterexample $Q$ minimizing $\card{Q}$.  
Now $Q = L(G)$ for a critical graph $G$.  Say $\chi(Q) = \chi'(G) = k + 1$.  So
$k = \max\set{\omega(Q),\ceil{\frac{5\Delta(Q)+3}{6}}}$ by the minimality of
$\card{Q}$.  By Corollary \ref{mainCor},
\[k+1 \le \max\set{\W(G), \ceil{\frac{5\Delta(Q)+3}{6}}},\]
so
\[\W(G) = \ceil{\frac{5\Delta(Q)+3}{6}} + 1 = \chi(Q).\]
Therefore $G$ is elementary and $k = \frac56\parens{\Delta(Q) + 1} + \beta$ for some $\beta \ge -\frac13$.  By Corollary~\ref{mainCorHelper}, $k \ge \Delta(G) + 1$.
Let $H$ be the underlying simple graph of $G$.
We may apply Corollary \ref{CriticalElementary56} to conclude that there is at most one $x \in V(G)$ with $d_H(x) \ge 3$.
Since $G$ is critical, $\delta(H) \ge 2$ and $H$ has no cut vertices. Hence $H$ is a cycle. 

Choose $t$ such that $\card{V(H)}=2t+1$.
Let $x_1,\ldots,x_{2t+1}$ denote the multiplicities of the edges in $G$, and let
$X=\sum_{i=1}^{2t+1}x_i$.  Since $G$ is elementary, we have
$\chi'(G)=\ceil{\frac{X}t}$.  
Let $Q=L(G)$ and let $v_i$ be a vertex of $Q$
corresponding to an edge of $G$ counted by $x_i$.  Now
$d_Q(v_i)=x_{i-1}+x_i+x_{i+1}-1$.  It suffices to show that there exists
$j\in[2t+1]$ such that $\frac{X}t\le \frac{5d_Q(v_j)+3}6$.  We will prove the
stronger statement that $\frac{X}t\le \frac{5\overline{d}+3}6$, where
$\overline{d}=\frac{1}{2t+1}\sum_{i=1}^{2t+1}d_Q(v_i)$.
Since $\frac{5\overline{d}+3}6 =\frac{5X}{2(2t+1)}-\frac13$, it suffices to have 
$\frac{5X}{2(2t+1)}-\frac13\ge \frac{X}t$.  Simplifying (for $t\ge 3$) gives $X\ge
\frac13(4t+10+\frac{20}{t-2})$.  Since $X\ge 2t+1$, this always holds when $t\ge
6$.  When $t=5$, it suffices to have $X\ge 13$.  When $t=4$, it suffices to have
$X\ge 12$, and when $t=3$, it suffices to have $X\ge 14$.  Suppose $t=5$ and
$X\le 12$.  
Now $\chi'(G)\le \ceil{\frac{12}5}=3\le\ceil{\frac{5(2)+3}6}$.
Suppose instead that $t=4$ and $X\le 11$.
Now $\chi'(G)\le \ceil{\frac{11}4}=3\le\ceil{\frac{5(2)+3}6}$.
Suppose instead that $t=3$ and $X\le 13$.
If $X\le 9$, then again 
$\chi'(G)\le \ceil{\frac{9}3}=3\le\ceil{\frac{5(2)+3}6}$.
So assume that $X\ge 10$, which implies that $\Delta(Q)\ge 4$.
First suppose that $X\le 12$.
Now $\chi'(G)\le \ceil{\frac{12}3}=4\le\ceil{\frac{5(4)+3}6}$.
So instead assume that $X=13$, which implies that $\Delta(Q)\ge 5$.
Now $\chi'(G)\le \ceil{\frac{13}3}=5\le\ceil{\frac{5(5)+3}6}$, as desired.
Finally, we consider $t=2$.  Now $\frac{5\overline{d}+3}6=\frac{X}2-\frac13$ and
$\frac{X}t=\frac{X}2$, so always $\frac{5\overline{d}+3}6<\frac{X}t$.  However,
$\ceil{\frac{5\overline{d}+3}6}=\ceil{\frac{X}2-\frac13}=\ceil{\frac{X}2}$, for
all integers $X$, which completes the proof.
\end{proof}

We suspect that our Main Theorem can be extended to the larger class of
quasi-line graphs (those for which the neighborhood of each vertex is covered by
two cliques).

\begin{conj}
If $Q$ is a quasi-line graph, then
\[\chi(Q)\le \max\set{\omega(Q),\frac{5\Delta(Q)+8}{6}}.\]
\end{conj}

\section{Strengthenings of Reed's Conjecture}
\label{sec:strong-reed}
In this section, we show how the results of Section~\ref{sec:short} 
imply Reed's Conjecture, as well as Local and Superlocal strengthenings of
Reed's Conjecture, for the class of line graphs.

Let $G$ be a graph.  The \emph{claw-degree}\aside{claw-degree} of $x \in V(G)$ is 
\[\dclaw{x} \DefinedAs \max_{\substack{S \subseteq N(x) \\ \card{S} = 3}}\frac14 \parens{d(x) + \sum_{v \in S} d(v)},\]
where $\dclaw{x} \DefinedAs 0$ when $\card{N(x)} \le 2$.
The \emph{claw-degree} of $G$ is 
\[\dclaw{G} \DefinedAs \max_{x \in V(G)} \dclaw{x}.\]
\begin{thm}
\label{EasyBound}
If $G$ is a graph, then
\[\chi'(G) \le \max\set{\W(G), \Delta(G) + 1, \ceil{\frac43\dclaw{G}}}.\]
\end{thm}
\begin{proof}
Suppose not and choose a counterexample $G$ with the fewest edges; note that $G$
is critical. 
Let $k=\chi'(G)-1$, so $k \ge \ceil{\frac43\dclaw{G}}$. 
By Theorem \ref{AllSpecialImpliesElementary}, $G$ has a long vertex $x$.
Choose $xy_1 \in E(G)$ and a $k$-edge-coloring $\vph$ of $G - xy_1$ such that
$\vph$ has a fan $F$ of length $3$ rooted at $x$ with leaves $y_1, y_2, y_3$.  
Since no two vertices of $F$ miss a common color,
\[2 + k - d(x) + \sum_{i \in \irange{3}} k-d(y_i) \le k,\]
and hence
\[
\frac{3k+2}{4}
\le 
\frac14\parens{d(x) + \sum_{i \in \irange{3}} d(y_i)} 
\le 
\dclaw{x}.
\]
This gives the contradiction
\[\ceil{\frac43\dclaw{G}} \le k \le \frac43\dclaw{G} - \frac23.\]
\aftermath
\end{proof}

Reed~\cite{Reed1998omega} conjectured that $\chi(G)\le
\ceil{\frac{\omega(G)+\Delta+1}2}$ for every
graph $G$.  This is the average of a trivial lower bound $\omega(G)$ and a
trivial upper bound $\Delta(G)+1$.  King~\cite{King-diss} conjectured the stronger
bound $\chi(G)\le \max_{v\in V(G)}\ceil{\frac{\omega(v)+d(v)+1}2}$, where
$\omega(v)$ is the size of the largest clique containing $v$; this bound is now known
to hold for many classes of graphs, including line graphs~\cite{CKPS}.  Here we
show that for line graphs this bound is an easy consequence of our more general
lemmas from Section~\ref{sec:short}.  A \emph{thickened cycle} is a multigraph
that has a cycle as its underlying simple graph.

\begin{thm}
\label{EasyBound2}
If $G$ is a critical graph that is not a thickened cycle, then
\[\chi'(G) \le \max\set{\Delta(G) + 1, \ceil{\frac43\dclaw{G}}}.\]
\end{thm}
\begin{proof}
The proof is in some ways similar to those of Lemmas~\ref{Slacked}
and~\ref{DegreeBoundedForMiddling}.  It consists largely of straightforward,
albeit tedious, algebraic manipulations.

Suppose the theorem is false and let $G$ be a counterexample. By Theorem
\ref{EasyBound}, $G$ is elementary.
Since $G$ is critical and elementary, Proposition~\ref{easy-prop} implies that
$\card{G}$ is odd and
\begin{equation}\label{eq11}
k = \frac{2(\size{G} - 1)}{\card{G} - 1}.
\end{equation}
Let $x \in V(G)$ with $\card{N(x)} \ge 3$. Put $M \DefinedAs \card{N(x)}$, 
\[P \DefinedAs \sum_{v \in N(x)} d_G(v),\]
\[S_2 \DefinedAs 2 + \sum_{v \in V(G) \setminus N(x)} \Delta(G) - d_G(v),\]
\[S_3 \DefinedAs k - (\Delta(G) + 1).\] 
Now
\begin{equation}\label{eq12}
2(\size{G} - 1) = \Delta(G)(\card{G} - M) - S_2 + P.
\end{equation}
Since 
\[\frac{2(\size{G} - 1)}{\card{G} - 1} = k = \Delta(G) + 1 + S_3,\]
using \eqref{eq12}, we get
\[P = (\card{G} - 1)(\Delta(G) + 1 + S_3) - \Delta(G)(\card{G} - M) + S_2,\]
which we rewrite as
\begin{equation}\label{eq13}
P = \Delta(G)(M-1) + \card{G} - 1 + S_2 + S_3(\card{G} - 1).
\end{equation}

Let $N(x) = \set{v_0, v_1, \ldots, v_{M-1}}$ and put
\[R \DefinedAs \sum_{i=0}^{M - 1} \frac13 \parens{d_G(x) + d_G(v_i) + d_G(v_{i+1}) + d_G(v_{i+2})},\]
where indices are taken modulo $M$.  Since $k \ge \frac43\dclaw{G}$, there is $S_4 \ge 0$ such that
\[Mk - S_4 = R = \frac{M}{3}d_G(x) + P.\]
Now substituting \eqref{eq13} we get
\[Mk = \frac{M}{3}d_G(x) + \Delta(G)(M-1) + \card{G} - 1 + S_2 + S_3(\card{G} - 1) + S_4.\]
Since $M\ge 3$ by our choice of $x$, 
solving for $M$ gives (for some $S_5 \ge 0$)
\[3 + S_5 = M =  \frac{S_2 + (S_3 + 1)(\card{G} - 1) + S_4 -\Delta(G)}{S_3 + 1
- \frac13d_G(x)}. \]
Hence
\[(3 + S_5)(S_3 + 1) - \parens{1 + \frac{S_5}{3}}d_G(x) = S_2 + (S_3 +
1)(\card{G} - 1) + S_4 -\Delta(G).\]
Rearranging terms gives
\[\parens{1 + \frac{S_5}{3}}d_G(x) = \Delta(G)  - (S_2 - 2) + \parens{4 + S_5 -
\card{G}}S_3 + \parens{2 + S_5 - \card{G}} - S_4.\]
Suppose $4 + S_5 - \card{G} \le 0$.  Now
\[\parens{1 + \frac{S_5}{3}}d_G(x) \le \Delta(G)  - (S_2 - 2) - 2.\]
By definition, $S_2 \ge 2 + \Delta(G) - d_G(x)$, so we have
\[\parens{1 + \frac{S_5}{3}}d_G(x) \le d_G(x) - 2,\]
a contradiction since $S_5 \ge 0$.  So, we must have $4 + S_5 - \card{G} > 0$, that is, 
\[\card{G} \le S_5 + 3 = \card{N(x)} \le \card{G} - 1,\] a contradiction.
\end{proof}
For a graph $Q$ and $r \in \IN \cup \set{\infty}$, put
\[\C_r(Q) \DefinedAs \setbs{X \subseteq V(Q)}{X \text{ is a maximal clique with } \card{X} < r \text { or } X \text { is a clique with } \card{X} = r}.\]  
Put
\[\gamma_r(Q) \DefinedAs \max_{X \in \C_r(Q)} \frac{1}{\card{X}}\sum_{v \in X} \frac{d(v) + \omega(v) + 1}{2}.\]
In terms of $\gamma_r$, the local version of Reed's conjecture for line graphs,
proved by Chudnovsky et al.~\cite{CKPS}, and the superlocal version, proved by
Edwards and King~\cite{EK-superlocal}, are the following two theorems.

\begin{thm}[Chudnovsky et al.]
If $Q$ is a line graph, then $\chi(Q) \le \ceil{\gamma_1(Q)}$.
\label{LocalReed}
\end{thm}

\begin{thm}[Edwards and King]
If $Q$ is a line graph, then $\chi(Q) \le \ceil{\gamma_2(Q)}$.
\label{KingEdwardsSuperLocal}
\end{thm}

Note that if $a, b \in \IN \cup \set{\infty}$ with $a \le b$, then $\gamma_a(Q) \ge \gamma_b(Q)$, so Theorem \ref{KingEdwardsSuperLocal} implies Theorem \ref{LocalReed}.  
Edwards and King~\cite{EK-superlocal} conjectured that $\gamma_\infty(Q)$ is an
upper bound on the fractional chromatic number for all graphs $Q$.

\begin{conj}[Edwards and King]
If $Q$ is any graph, then $\chi_f(Q) \le \gamma_\infty(Q)$.
\label{EKconj}
\end{conj}

A graph parameter $f$ is \emph{monotone} if $f(G)\ge f(H)$ whenever $H$ is an
induced subgraph of $G$.
Unfortunately, $\gamma_\infty$ is not monotone and hence not very induction-friendly.
\begin{lem}
$\gamma_1$ and $\gamma_2$ are monotone, but $\gamma_r$ is not monotone for all
$r \ge 3$.
\end{lem}
\begin{proof}
The first statement is clear.
For the second statement, let $Q$ be the graph in Figure~\ref{fig:NotMonotone}.
 For $r \ge 3$, we have $\gamma_r(Q -x) = \frac{11}{2} > \frac{16}{3} = \gamma_r(Q)$.
\end{proof}

\begin{figure}
\centering
\begin{tikzpicture}[scale = 10]
\tikzstyle{VertexStyle} = []
\tikzstyle{EdgeStyle} = []
\tikzstyle{labeledStyle}=[shape = circle, minimum size = 6pt, inner sep = 1.2pt, draw]
\tikzstyle{unlabeledStyle}=[shape = circle, minimum size = 6pt, inner sep = 1.2pt, draw, fill]
\Vertex[style = labeledStyle, x = 0.700, y = 0.900, L = \tiny {$x$}]{v0}
\Vertex[style = unlabeledStyle, x = 0.600, y = 0.750, L = \tiny {}]{v1}
\Vertex[style = unlabeledStyle, x = 0.800, y = 0.750, L = \tiny {}]{v2}
\Vertex[style = unlabeledStyle, x = 0.300, y = 0.300, L = \tiny {}]{v3}
\Vertex[style = unlabeledStyle, x = 1.100, y = 0.300, L = \tiny {}]{v4}
\Vertex[style = unlabeledStyle, x = 0.600, y = 0.600, L = \tiny {}]{v5}
\Vertex[style = unlabeledStyle, x = 0.800, y = 0.600, L = \tiny {}]{v6}
\Vertex[style = unlabeledStyle, x = 0.550, y = 0.550, L = \tiny {}]{v7}
\Vertex[style = unlabeledStyle, x = 0.850, y = 0.550, L = \tiny {}]{v8}
\Vertex[style = unlabeledStyle, x = 0.500, y = 0.500, L = \tiny {}]{v9}
\Vertex[style = unlabeledStyle, x = 0.900, y = 0.500, L = \tiny {}]{v10}
\Vertex[style = unlabeledStyle, x = 0.450, y = 0.450, L = \tiny {}]{v11}
\Vertex[style = unlabeledStyle, x = 0.950, y = 0.450, L = \tiny {}]{v12}
\Vertex[style = unlabeledStyle, x = 0.400, y = 0.400, L = \tiny {}]{v13}
\Vertex[style = unlabeledStyle, x = 1.000, y = 0.400, L = \tiny {}]{v14}
\Vertex[style = unlabeledStyle, x = 0.350, y = 0.350, L = \tiny {}]{v15}
\Vertex[style = unlabeledStyle, x = 1.050, y = 0.350, L = \tiny {}]{v16}
\Edge[label = \tiny {}, labelstyle={auto=right, fill=none}](v1)(v0)
\Edge[label = \tiny {}, labelstyle={auto=right, fill=none}](v1)(v2)
\Edge[label = \tiny {}, labelstyle={auto=right, fill=none}](v1)(v3)
\Edge[label = \tiny {}, labelstyle={auto=right, fill=none}](v1)(v5)
\Edge[label = \tiny {}, labelstyle={auto=right, fill=none}](v1)(v7)
\Edge[label = \tiny {}, labelstyle={auto=right, fill=none}](v1)(v9)
\Edge[label = \tiny {}, labelstyle={auto=right, fill=none}](v1)(v11)
\Edge[label = \tiny {}, labelstyle={auto=right, fill=none}](v1)(v13)
\Edge[label = \tiny {}, labelstyle={auto=right, fill=none}](v1)(v15)
\Edge[label = \tiny {}, labelstyle={auto=right, fill=none}](v2)(v0)
\Edge[label = \tiny {}, labelstyle={auto=right, fill=none}](v2)(v4)
\Edge[label = \tiny {}, labelstyle={auto=right, fill=none}](v2)(v6)
\Edge[label = \tiny {}, labelstyle={auto=right, fill=none}](v2)(v8)
\Edge[label = \tiny {}, labelstyle={auto=right, fill=none}](v2)(v10)
\Edge[label = \tiny {}, labelstyle={auto=right, fill=none}](v2)(v12)
\Edge[label = \tiny {}, labelstyle={auto=right, fill=none}](v2)(v14)
\Edge[label = \tiny {}, labelstyle={auto=right, fill=none}](v2)(v16)
\Edge[label = \tiny {}, labelstyle={auto=right, fill=none}](v4)(v3)
\Edge[label = \tiny {}, labelstyle={auto=right, fill=none}](v6)(v5)
\Edge[label = \tiny {}, labelstyle={auto=right, fill=none}](v8)(v7)
\Edge[label = \tiny {}, labelstyle={auto=right, fill=none}](v10)(v9)
\Edge[label = \tiny {}, labelstyle={auto=right, fill=none}](v12)(v11)
\Edge[label = \tiny {}, labelstyle={auto=right, fill=none}](v14)(v13)
\Edge[label = \tiny {}, labelstyle={auto=right, fill=none}](v16)(v15)
\end{tikzpicture}
\caption{A graph $Q$ where $\gamma_r(Q - x) > \gamma_r(Q)$ for all $r \ge 3$.}
\label{fig:NotMonotone}
\end{figure}
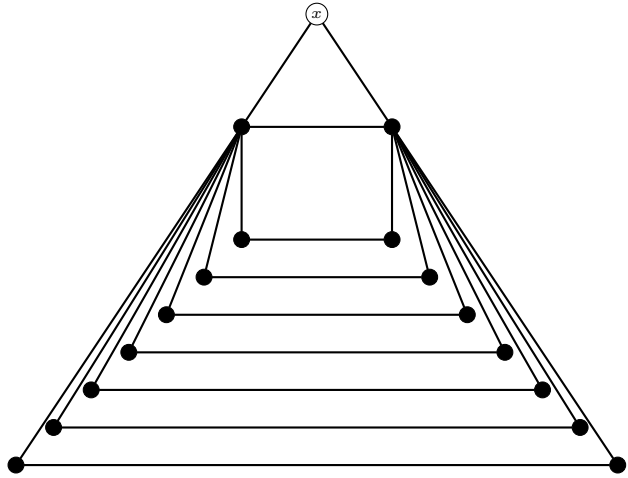

As a possible approach to Conjecture~\ref{EKconj}, we define the following
variation on $\gamma_{\infty}(G)$, which should be more induction-friendly.
For $r\in \IN \cup \set{\infty}$, put
\[\tilde{\gamma_r}(Q) \DefinedAs \max_{H \subseteq Q} \gamma_r(H).\]

\begin{conj}\label{KingEdwardsReplacement}
If $Q$ is any graph, then $\chi_f(Q) \le \tilde{\gamma}_\infty(Q)$.
\end{conj}

\begin{conj}\label{SuperDuperDuperLocalReed}
If $Q$ is a line graph, then $\chi(Q) \le \ceil{\tilde{\gamma}_\infty(Q)}$.
\end{conj}

\begin{thm}
Conjecture \ref{SuperDuperDuperLocalReed} follows from Conjecture \ref{KingEdwardsReplacement} and the Goldberg-Seymour Conjecture.
\end{thm}
\begin{proof}
Assume that Conjecture \ref{KingEdwardsReplacement} and the Goldberg--Seymour
Conjecture are true and Conjecture \ref{SuperDuperDuperLocalReed} is false. 
Let $Q$ be a minimal counterexample, and say $Q = L(G)$.   
Since Conjecutre~\ref{KingEdwardsReplacement} is true,
$\chi_f(Q)\le\tilde{\gamma}_\infty(Q)$.
Since the Goldberg--Seymour Conjecture is true,
$\chi(Q)\le\max\{\Delta(G)+1,\ceil{\chi_f(Q)}\} \le
\max\{\Delta(G)+1,\ceil{\tilde{\gamma}_\infty(Q)}\}$.
Since $Q$ is a counterexample
$\ceil{\tilde{\gamma}_\infty(Q)}
<\chi(Q)\le\Delta(G)+1\le \omega(Q)+1$.
This implies that 
$\ceil{\tilde{\gamma}_\infty(Q)}\le\omega(Q)$.  
However, also 
$\tilde{\gamma}_\infty(Q)\ge \omega(Q)$, by taking $X$ in the definition of 
$\tilde{\gamma}_\infty$ to be a maximum clique.  Further, since $Q$ is connected
by minimality, the inequality is strict if $X\subsetneq Q$.  So $Q$ is a clique.
But this yields the contradiction
$\omega(Q) = \ceil{\tilde{\gamma}_\infty(Q)} < \chi(Q)$.  
%
\end{proof}

We prove the next bound in the sequence begun by Theorems~\ref{LocalReed}
and~\ref{KingEdwardsSuperLocal}.  
\begin{thm}\label{SuperDuperLocalReed}
If $Q$ is a line graph, then $\chi(Q) \le \ceil{\tilde{\gamma}_3(Q)}$.
\end{thm}

Before proving Theorem~\ref{SuperDuperLocalReed}, we prove the following lemma,
which aids in the proof of Theorem~\ref{SuperDuperLocalReed}.

\begin{lem}\label{SuperDuperLocalReedHelper}
Let $Q = L(G)$ where $G$ is a critical graph. If $G$ is not a thickened cycle, then $\chi(Q) \le \ceil{\gamma_3(Q)}$.
\end{lem}
\begin{proof}
Suppose $G$ is not a thickened cycle. For $uv\in E(G)$, put 
\[f(uv) \DefinedAs \max\{d_G(u)+\frac12(d_G(v)-\mu(uv)),d_G(v)+\frac12(d_G(u)-\mu(uv)\}.\]
For $uv \in E(G)$, we have
\begin{align*}
f(uv) &= \frac{d_G(u) + d_G(v) - \mu(uv) + \max\set{d_G(u), d_G(v)}}{2}\\
&\le \frac{d_Q(uv) + \omega(uv) + 1}{2}.
\end{align*}
Since $G$ is critical, we have $\card{N(v)} \ge 2$ for all $v \in V(G)$.  
Since $G$ is not a thickened cycle, we may choose $x\in V(G)$ and $S\subseteq
N(x)$ with $\card{S} = 3$ such that $x$ and $S$ achieve maximality in the definition of $d_{claw}(G)$. 
Say $S = \set{v_1, v_2, v_3}$.  Then
\begin{align*}
\ceil{\frac43\dclaw{G}} &= \ceil{\parens{\frac43}\parens{\frac14}\parens{d_G(x)+\sum_{i \in \irange{3}} d_G(v_i)}} \\
&\le \ceil{\frac13\sum_{i \in \irange{3}} d_G(v_i) + \frac12(d_G(x) - \mu(xv_i))}\\
&\le \ceil{\frac13\sum_{i \in \irange{3}} f(xv_i)}\\
&\le \ceil{\frac13\sum_{i \in \irange{3}} \frac{d_Q(xv_i) + \omega(xv_i) + 1}{2}}\\
&\le \ceil{\gamma_3(Q)}.
\end{align*}
By Theorem \ref{EasyBound2}, we have
\begin{equation}\label{EasyBoundEquation}
\chi(Q) \le \max\set{\Delta(G) + 1, \ceil{\gamma_3(Q)}}.
\end{equation}

Let $M \subseteq V(Q)$ be a maximum clique in $Q$.  
Since $G$ is not a thickened cycle, $\card{M} \ge 3$, so we can choose $X
\subseteq M$ with $\card{X} = 3$ maximizing
\[\frac{1}{3}\sum_{v \in X} \frac{d(v) + \omega(v) + 1}{2}.\]
We have
\begin{align*}
\gamma_3(Q) &\ge \frac{1}{3}\sum_{v \in X} \frac{d(v) + \omega(v) + 1}{2}\\
&\ge \frac{1}{\card{M}}\sum_{v \in M} \frac{d(v) + \omega(v) + 1}{2}\\
&\ge \omega(Q) + \sum_{v \in M} \frac{d(v) + 1 - \omega(v)}{2}.
\end{align*}
If $V(Q) = M$, then $\ceil{\gamma_3(Q)} = \omega(Q) = \Delta(Q) = \chi(Q)$, as desired.
Otherwise, some $v \in M$ has $d(v) \ge \omega(Q)$ and hence $\ceil{\gamma_3(Q)} \ge \omega(Q) + 1 \ge \Delta(Q) + 1$.  Using 
\eqref{EasyBoundEquation}, this gives $\chi(Q) \le \ceil{\gamma_3(Q)}$.
\end{proof}

\begin{proof}[Proof of Theorem \ref{SuperDuperLocalReed}]
Suppose Theorem \ref{SuperDuperLocalReed} is false, and choose a counterexample
$Q$ minimizing $\card{Q}$.  Say $Q = L(G)$.  The minimality of $\card{Q}$ and
monotonicity of $\tilde{\gamma}_3$ imply that $G$ is critical.  So Lemma
\ref{SuperDuperLocalReedHelper} gives
a contradiction unless $G$ is a thickened cycle.  

Suppose that $G$ is a thickened cycle.  
We first show that 
$\ceil{\tilde{\gamma}_3(Q)}\ge \Delta(G)+1$, which will show that $G$ is
elementary.
Since the theorem is trivially true for cycles we can assume $\Delta(G)\ge 3$.
We take $Y$ to be any clique of size 3\ in $Q$ corresponding to edges incident
to a maximum degree vertex of $G$. Now $Y$ witnesses that
$\tilde{\gamma}_3(Q)>\Delta(G)$, so $\ceil{\tilde{\gamma}_3(Q)}\ge
\Delta(G)+1$, as desired.  
Since $Q$ is a counterexample, we must have $\chi(Q)>\Delta(G)+1$.
Since $G$ is a thickened cycle, trivially every vertex is short.  Thus,
Theorem~\ref{AllSpecialImpliesElementary} implies that $G$ is elementary, i.e.,
$\chi'(G)=\W(G)$.  In fact, since $G$ is critical,
$\chi'(G)=\ceil{\frac{2\size{G}}{\card{G}-1}}$.

Since $G$ is elementary, $\card{G}$ is
odd, so say $\card{G}=2t+1$.  Denote the vertices of $G$ by $\{v_1,\ldots,
v_{2t+1}\}$, and let $x_i = \mu(v_iv_{i+1})$ for each $i\in [2t+1]$.  Let
$X=\sum_{i=1}^{2t+1}x_i$.  First suppose there exists $j$ such that $x_j=1$.
So there exists $e\in E(G)$ such that $G-e$ is bipartite.  Thus,
$\chi'(G-e)=\Delta(G-e)\le \Delta(G)$, so $\chi'(G)\le \Delta(G)+1$.
Since $\ceil{\tilde{\gamma}_3(Q)}\ge \Delta(G)+1$, as shown above,
the theorem holds.

Now assume instead that $x_i\ge 2$ for all $i\in[2t+1]$.
For each $i$, let $u^1_i$ and $u^2_i$ be vertices of $Q$ corresponding to edges
counted by $x_i$.  Note that $\frac{d(u^j_i)+\omega(u^j_i)}2 \ge
\frac{x_{i-1}+2x_i+2x_{i+1}}2$, for each $i\in[2t+1]$ and each $j\in[2]$.
Now averaging over $\{u^1_i, u^1_{i+1}, u^2_{i+1}\}$
gives
$\tilde{\gamma}_3(Q) \ge \frac16\left(x_{i-1}+4x_i+6x_{i+1}+4x_{i+2}\right).$
When we average over all $i \in [2t+1]$, we get $\tilde{\gamma}_3(Q) \ge
\frac{5X}{4t+2}$.  Since $G$ is elementary, we know $\chi'(G) = \ceil{\frac{X}t}$.  
Now the theorem holds since $\frac1t \le \frac5{4t+2}$, whenever $t \ge 2$.
%
\end{proof}  
\bibliographystyle{plain}
\bibliography{GraphColoring1}

\def\cprime{$'$}
\begin{thebibliography}{10}

\bibitem{andersen1977edge}
L.D. Andersen.
\newblock On edge-colorings of graphs.
\newblock {\em Math. Scand}, 40:161--175, 1977.

\bibitem{CKPS}
M.~Chudnovsky, A.D. King, M.~Plumettaz, and P.~Seymour.
\newblock A local strengthening of {R}eed's {$\omega$}, {$\Delta$}, {$\chi$}
  conjecture for quasi-line graphs.
\newblock {\em SIAM J. Discrete Math.}, 27(1):95--108, 2013.

\bibitem{diestel2010}
R.~Diestel.
\newblock {\em {Graph Theory}}.
\newblock Springer-Verlag, Heidelberg, 4 edition, 2010.

\bibitem{EK-superlocal}
K.~Edwards and A.D. King.
\newblock A superlocal version of {R}eed's conjecture.
\newblock {\em Electron. J. Combin.}, 21(4):Paper 4.48, 18, 2014.

\bibitem{goldberg1973}
M.K. Goldberg.
\newblock Multigraphs with a chromatic index that is nearly maximal.
\newblock {\em Diskret. Analiz}, (23):3--7, 72, 1973.
\newblock A collection of articles dedicated to the memory of Vitali{\u\i}
  Konstantinovi{\v{c}} Korobkov.

\bibitem{GoldbergJGT}
M.K. Goldberg.
\newblock Edge-coloring of multigraphs: Recoloring technique.
\newblock {\em Journal of Graph Theory}, 8(1):123--137, 1984.

\bibitem{King-diss}
A.D. King.
\newblock {\em Claw-free graphs and two conjectures on omega, {D}elta, and
  chi}.
\newblock ProQuest LLC, Ann Arbor, MI, 2009.
\newblock Thesis (Ph.D.)--McGill University (Canada).

\bibitem{rabern2011strengthening}
L.~Rabern.
\newblock {A strengthening of Brooks' Theorem for line graphs}.
\newblock {\em Electron. J. Combin.}, 18(p145):1, 2011.

\bibitem{Reed1998omega}
B.~Reed.
\newblock {$\omega$, $\Delta$, and $\chi$}.
\newblock {\em Journal of Graph Theory}, 27(4):177--212, 1998.

\bibitem{seymour1979a}
P.D. Seymour.
\newblock On multicolourings of cubic graphs, and conjectures of {F}ulkerson
  and {T}utte.
\newblock {\em Proc. London Math. Soc. (3)}, 38(3):423--460, 1979.

\bibitem{seymour1979b}
P.D. Seymour.
\newblock Some unsolved problems on one-factorizations of graphs.
\newblock In J.A. Bondy and U.S.R. Murty, editors, {\em {Graph Theory and
  Related Topics (Proceedings Conference, Waterloo, Ontario, 1977)}}, pages
  367--368. 1979.

\bibitem{SSTF}
M.~Stiebitz, D.~Scheide, B.~Toft, and L.M. Favrholdt.
\newblock {\em Graph edge coloring}.
\newblock Wiley Series in Discrete Mathematics and Optimization. John Wiley \&
  Sons, Inc., Hoboken, NJ, 2012.
\newblock Vizing's theorem and Goldberg's conjecture, With a preface by
  Stiebitz and Toft.

\end{thebibliography}

\end{document}